\documentclass[11pt]{amsart}

\usepackage{amssymb}    % use all AMS fonts
\usepackage{amsmath}    % include some AMS-LaTeX functionality
\usepackage{amsthm}
\usepackage[dvipsnames]{xcolor}
\usepackage{tikz-cd}
\usepackage{placeins}
\usepackage{amscd}

\usepackage{geometry}
\usepackage{palatino}
\usepackage{mathpazo}

\usepackage{mathrsfs}

\usepackage{bbm}

\usepackage{upgreek}

\usepackage{caption}

\RequirePackage[colorlinks=true,citecolor=blue,urlcolor=blue]{hyperref}

\graphicspath{{figures/}} %Setting the graphicspath

\usepackage{url}
\makeatletter
\def\url@leostyle{%
  \@ifundefined{selectfont}{\def\UrlFont{\sf}}{\def\UrlFont{\scriptsize\ttfamily}}}
\makeatother
\newtheorem{theorem}{Theorem}[section]
\newtheorem{mainthm}{Theorem}
\newtheorem{proposition}[theorem]{Proposition} 
\newtheorem{corollary}[theorem]{Corollary} 
\newtheorem{lemma}[theorem]{Lemma}

\theoremstyle{definition}
\newtheorem{definition}[theorem]{Definition}

\newtheorem{remark}[theorem]{Remark}

\DeclareMathOperator{\hull}{Hull}
\DeclareMathOperator{\rad}{Rad}

\DeclareMathOperator{\tv}{TV}

\DeclareMathOperator{\pr}{pr}

\DeclareMathOperator{\op}{op}

\DeclareMathOperator{\diam}{diam}

\DeclareMathOperator{\rmin}{R_{min}}

\DeclareMathOperator{\h}{H}

\def\rrr{\mathbf{R}}

\newcommand{\ckrl}{\mathfrak{C}^k_{\rmin, \mathbf{L}}}
\def\curlyc{\mathscr{C}}

\def\tnorm{\|T\|_{\op}}
\def\snorm{\|S\|_{\op}}

\def\co{\colon\thinspace}

\def\ex{\mathbf{X}}
\def\ey{\mathbf{Y}}

\def\tildex{\widetilde{\mathbf{X}}}

\def\em{\mathbf{M}}
\def\tildem{\widetilde{\mathbf{M}}}
\def\hatem{\widehat{\mathbf{M}}}

\def\rmin{R_{\mathrm{min}}}
\def\rwfs{R_{\mathrm{wfs}}}

\def\Pi{\mathrm{I\,\!I}}

\def\diff {\mathrm{d}}

\renewcommand{\leq}{\leqslant}
\renewcommand{\geq}{\geqslant}

\newcommand{\spacing}[1]{\renewcommand{\baselinestretch}{#1}\large\normalsize}

\newcommand{\Rm}{R_{\min}}
\newcommand{\bL}{\bold L}
\newcommand{\cC}{\curlyc}
\newcommand{\cP}{\mathscr{P}}
\newcommand{\cM}{\mathscr{M}}
\newcommand{\cH}{\mathscr{H}}

\newcommand{\bbE}{\mathbf{E}}
\newcommand{\bbR}{\mathbf{R}}
\newcommand{\ind}{\mathbbm{1}}
\newcommand{\TV}{\operatorname{TV}}
\newcommand{\vol}{\operatorname{vol}}
\newcommand{\loc}{\operatorname{loc}}
\newcommand{\hloc}[1]{h_{#1}^{\loc}(p,r_1,r_2 ; t)}
\spacing{1.2}

\begin{document}

\title[Estimating the reach of a manifold]{Estimating the reach of a manifold\\ via its convexity defect function}

\author[Berenfeld, Harvey, Hoffmann, Shankar]{Cl\'{e}ment Berenfeld, John Harvey, Marc Hoffmann, Krishnan Shankar}
\address{Cl\'ement Berenfeld, Universit\'e Paris-Dauphine PSL, CEREMADE, Place du Mar\'echal de Lattre de Tassigny, 75016 Paris, France}
\email{berenfeld@ceremade.dauphine.fr}
\address{John Harvey, Department of Mathematics, Swansea University, Fabian Way, Swansea, SA1 8EN, U.K.}
\email{j.m.harvey@swansea.ac.uk}
\address{Marc Hoffmann, Universit\'e Paris-Dauphine PSL, CEREMADE, Place du Mar\'echal de Lattre de Tassigny, 75016 Paris, France}
\email{hoffmann@ceremade.dauphine.fr}
\address{Krishnan Shankar, National Science Foundation, 2415 Eisenhower Avenue, Alexandria, VA 22314, U.S.A.}
\email{Krishnan.Shankar-1@ou.edu}

\date{\today}

\begin{abstract}
The reach of a submanifold is a crucial regularity parameter for manifold learning and geometric inference from point clouds. This paper relates the reach of a submanifold to its convexity defect function.
Using the stability properties of convexity defect functions, along with some new bounds and the recent submanifold estimator of Aamari and Levrard [Ann. Statist. \textbf{47} 177-–204 (2019)], 
an estimator for the reach is given.
A uniform expected loss bound over a $\curlyc^k$ model is found. Lower bounds for the minimax rate for estimating the reach over these models are also provided. 
The estimator almost achieves these rates in the $\curlyc^3$ and $\curlyc^4$ cases, with a gap given by a logarithmic factor. 
\end{abstract}
\maketitle



\section{Introduction}

\subsection{Motivation} The reach of a submanifold $\em \subseteq \rrr^D$ is a geometric invariant which measures how tightly the submanifold folds in on itself. Dating back to Federer \cite{federer}, it encodes both local curvature conditions as well as global `bottlenecks' arising from two regions of the manifold that are far apart in the manifold's intrinsic metric but are close in the ambient Euclidean metric.  The reach is a key regularity parameter in the estimation of other geometric information. Methods and algorithms from topological data analysis often use the reach as a `tuning parameter'. The correctness of their results depends on setting this parameter correctly. 

Statistical inference from point clouds has become an active area. In a probabilistic framework,  a {\it reach condition}, meaning that the reach of the submanifold under study is bounded below, is usually necessary in order to obtain minimax inference results in manifold learning. These include: homology inference \cite{NSW, Bal12}, curvature \cite{AL}, reach estimation itself \cite{Reach} as well as manifold estimation \cite{Gal12, KRW16, AL}.
In this context, there is a risk of algorithms 
being applied as ‘black boxes’ without attention to their underlying assumptions. 
Efficient reach estimation would be a vital addition to this field, providing a so-called sanity test of other results. 

In this direction, Aamari, Kim {\it et al.}\ paved the way: in \cite{Reach}, under some specific assumptions, an estimator of the reach has been proposed and studied when the observation is an $n$-sample of a smooth probability distribution supported on an unknown $d$-dimensional submanifold $\em$ of a Euclidean space $\rrr^D$ together with the tangent spaces at each sampled point. For certain types of $\curlyc^3$-regularity models, the estimator, based on a representation of the reach in terms of points of $\em$ and its tangent spaces (Theorem 4.18 in \cite{federer}) achieves the rate $n^{-2/(3d-1)}$. A lower bound for the minimax rate of convergence is given by $n^{-1/d}$.  In the special case when the reach of $\em$ is attained at a bottleneck, the algorithm in \cite{Reach} achieves this rate.  However, in general, one does not know whether this condition is satisfied {\it a priori}.

In this paper, we continue the study of reach estimation by taking a completely different route: we use the relationship between the reach of a  submanifold of $\rrr^D$ and its {\it convexity defect function}. This function was introduced by Attali, Lieutier and Salinas in \cite{ALS} and measures how far a (bounded) subset $\ex \subseteq \rrr^D$ is from being convex at a given scale. It is a powerful geometric tool that has other applications such as manifold reconstruction, see the recent work by Divol \cite{Di}. By establishing certain new quantitative properties of the convexity defect function of a submanifold $\em \subseteq \rrr^D$ that relate to both its curvature and bottleneck properties,  we show that the convexity defect function can be used to compute the reach of a submanifold. From this we obtain a method which transforms an estimator of $\em$, along with information on its error, into a new estimator of the reach.

The recent results of Aamari and Levrard in \cite{AL} provide an estimator of $\em$ which is optimal, to within logarithmic terms. Transforming this into an estimator of the reach, we obtain new convergence results over  general $\curlyc^k$-regularity models ($k \geq 3$). These rates improve upon the previous work of \cite{Reach}. By establishing lower bounds for the minimax rates of convergence, we prove that our results are optimal up to logarithmic terms in the cases $k=3$ and $k=4$.

\subsection{Main results} \label{sec: teaser}

We present here one of several possible definitions of the reach. 
Given a submanifold $\em \subseteq \rrr^D$, consider its $\delta$-offset given by the open set $\em^{\delta} \subseteq \rrr^D$, where
$$
\em^{\delta} = \bigcup_{p\in \em} B_{\delta}(p).
$$
Here $B_{\delta}(p)$ denotes the open Euclidean ball centered at $p$ with radius $\delta$. For small enough $\delta$ (a uniform choice for such $\delta$ exists in general only when $\em$ is compact), one has has the property that for all $y\in \em^{\delta} \setminus \em$, there is a unique straight line from $y$ to a point in $\em$ realizing the distance from $y$ to $\em$. In other words, the metric projection $\uppi \co \em^{\delta} \to \em$ is well defined.
 
\begin{definition}[Federer \cite{federer}] \label{def: reach fed}
	The \textit{reach} of a submanifold $\em$ is 
	$$\sup \left\lbrace \delta \geq 0 \colon \text{The nearest point projection } \uppi \co \em^{\delta} \to \em \text{ is well defined} \right\rbrace.$$
	We denote the reach by $R(\em)$ or simply $R$ when the context is clear. 
	\end{definition}

 Other equivalent characterizations of the reach exist. For example, in Section \ref{sec: upp bound defect} below, we use the characterization from  Theorem 4.18 in \cite{federer}. More recently Theorem 1 in \cite{boisssonatetal} defined the reach in terms of the metric distortion.

Our main results are obtained for a statistical model which imposes certain standard regularity conditions on the manifolds being considered, requires that they are compact and connected, and also imposes conditions on the distributions being considered which have support on those manifolds. The set of distributions satisfying these constraints on $\curlyc^k$ manifolds is denoted in the results below by $\mathscr{P}^k$ and these constraints are elaborated upon in Sections \ref{geomframework} and \ref{sec: upperbounds}.

\begin{mainthm}\label{main1}
	For $d$-dimensional submanifolds of regularity $\curlyc^k$ with $k \geq 3$, and for sufficiently large $n$, there exists an estimator $\widehat R$ explicitly constructed in Section \ref{sec: upperbounds} below that satisfies
	$$
	\sup_{P \in \mathscr{P}^k} \mathbf{E}_{P^{\otimes n}} \big[\big| \widehat R - R \big|\big] \leq C
	\begin{cases}
	\displaystyle	  \left( \frac{\log(n)}{n-1} \right) ^{1/d} & k =3 \\
	\displaystyle        \left( \frac{\log(n)}{n-1} \right) ^{k/(2d)} & k\geq 4,
	\end{cases}
$$
	where $\widehat R$ denotes an estimator of the reach $R = R(\em)$ constructed from an $n$-sample $(X_1,\ldots, X_n)$ of independent random variables with common distribution $P \in \mathscr{P}^k$. The quantity $C>0$ depends on $d$, $k$ and the regularity parameters that define the class $\mathscr{P}^k$ and the notation $\mathbf{E}_{P^{\otimes n}}[\cdot]$ refers to the expectation operator under the distribution $P^{\otimes n}$ of the $n$-sample $(X_1,\ldots, X_n)$. 
\end{mainthm}

We also provide a lower bound for the minimax convergence rate. In case $k=3,4$, our estimators are almost optimal, with a gap given by a $\log(n)$ factor.

\begin{mainthm}\label{main2}
	For certain values of the regularity parameters (depending only on $d$ and $k$), then
	$$\inf_{\widehat{R}}\sup_{P \in \mathscr{P}^k} \mathbf{E}_{P^{\otimes n}} \big[\big| \widehat{R} -R\big|\big] \geq c n^{-(k-2)/d},$$
	where the infimum is taken over all the estimators
	$\widehat{R} = \widehat{R}(X_1 ,\ldots ,X_n )$ and $c>0$ depends on $d$, $k$ and the regularity parameters.
\end{mainthm}

These results are of an entirely theoretical nature. The question of practical implementation remains, although it is not of primary interest for the paper. Starting from a point cloud $X_1,\ldots, X_n$, one may implement the following protocol: 
\begin{itemize}
\item Estimate $\em$ from the data  $X_1,\ldots, X_n$ by the best available manifold reconstruction method $\widehat{\em}$, or, indeed, by any other method.
\item Compute $h_{\widehat{\em}}$ (Definition \ref{def: convexity defect function}) and derive $\widehat R$ thanks to Definition \ref{def: final est}.
\end{itemize}
The only inputs are $\widehat{\em}$ and the regularity parameters that define the class $\mathscr{P}^k$. It is a common practice in statistics to assume some prior knowledge of the class in order to constrain the problem. However, the quantities $C_{d,k,R_{\min}}$ and $C$ in Theorem \ref{th: upper glob reach} are unknown, which creates  difficulties in deriving the accuracy of the estimator $\widehat R$ and, for example, calculating a confidence interval. This is common to every minimax result and could in practice be treated by estimating the variance of the estimator via any conventional computational method such as the bootstrap \cite{Boot}. A more detailed discussion lies outside the scope of the present paper.

\subsection{Organization of the paper}\hfill

The paper is divided into two halves: a first half that is mainly geometric in flavor and a second half which employs mainly statistical techniques.  To that end Sections \ref{sec: geo of the reach}, \ref{geomframework} and \ref{s:cdf} describe the geometric setting of this paper in some detail, Section \ref{sec: esti} discusses the approximation of the reach in a deterministic setting, while Sections \ref{sec: upperbounds} and \ref{sec: lower bounds} are devoted to showing that the new algorithm proposed to estimate the reach achieves the rates stated in Theorem \ref{main1} and to the proof of the lower bound for the minimax rate stated in Theorem \ref{main2}.
\smallskip

Section \ref{sec: geo of the reach}:  We elaborate on the geometry of the reach. We recall a dichotomy due to Aamari, Kim {\it et al.} \cite{Reach} in Theorem \ref{reachdichotomy} and we study in particular the distinction between {\it global reach} or {\it weak feature size} in Definition \ref{def: glob reach} and the {\it local reach} in Definition \ref{def: loc reach}, according to the terminology of \cite{Reach}. This is not apparent in the classical Definition \ref{def: reach fed} of Federer.
\smallskip

Section \ref{geomframework}:  A geometrical framework is given for studying reach estimation. We describe precisely 
a class $\ckrl$ of submanifolds, following Aamari and Levrard \cite{AL}. Manifolds $\em$ in this class admit a local parametrization at all points $p \in \em$ by the tangent space $T_p\em$, which is the inverse of the projection to the tangent space and satisfies certain $\curlyc^k$ bounds. 
\smallskip

Section \ref{s:cdf}:  This section is devoted to the study of the {\it convexity defect function} $h_\em$ of $\em$ as introduced in \cite{ALS} and its properties. We show how the local reach can be calculated from the values of $h_\em$ near the origin in Proposition \ref{taylorpoly} and how the weak feature size (the global reach) appears as a discontinuity point of $h_\em$ whenever it is smaller than the local reach. This is done by proving an upper bound on $h_\em$ in Proposition \ref{p:newbound}. Proposition \ref{taylorpoly} and \ref{p:newbound} are central to the results of the paper.
\smallskip

Section \ref{sec: esti}:  When we attempt to estimate the reach in later sections, we will not know $\em$ exactly. Instead, we will know it up to some statistical error coming from an estimator. Propositions \ref{localloss} and \ref{globalloss} give approximations of the local reach and the weak feature size, respectively, calculated from some proxy $\widetilde \em$. The errors of the approximations are given in terms of the Hausdorff distance $H(\em,\widetilde \em)$. 
\smallskip

Section \ref{sec: upperbounds}:  Building on the definitions in Section \ref{geomframework}, a statistical framework is described within which we study reach estimation in a minimax setting. This defines a class $\mathscr{P}^k$ of admissible distributions $P$ over their support $\em$, the submanifold of interest, which belongs to the class $\ckrl$. To apply the results of the previous section, we may use the Aamari--Levrard estimator \cite{AL} $\hatem$ of $\em$ from a sample $(X_1,\ldots, X_n)$ as the proxy $\widetilde \em$ for $\em$. This estimator is almost optimal over the class $\mathscr{P}^k$.  This yields estimators of the local reach and finally of the reach $R(\em)$ in Section \ref{sec: upperbounds}.  We then prove the upper bounds announced in Theorem \ref{main1} above in Theorems \ref{upper local condition}--\ref{th: upper glob reach}. 
\smallskip

Section \ref{sec: lower bounds}: Using the classical Le Cam testing argument we obtain minimax lower bounds as announced in Theorem \ref{main2}.

\section{Geometry of the reach} \label{sec: geo of the reach}

The reach of a submanifold $\em$, which we will denote by $R(\em)$, or simply $R$, is an unusual invariant. Definition \ref{def: reach fed} conceals  what is almost a dichotomy -- the reach of a submanifold can be realised in two very different ways. This is made precise by the following result.

\begin{theorem}\cite[Theorem 3.4]{Reach} \label{reachdichotomy} Let $\em \subseteq \rrr^D$ be a compact submanifold with reach $R(\em) > 0$. At
least one of the following two assertions holds.
\begin{itemize}
	\item \textup{(Global case)} $\em$ has a bottleneck, that is, there exist $q_1, q_2 \in \em$ such that $(q_1 + q_2)/2 \in Med(\em)$ and $\left \| q_1 - q_2 \right \| = 2 R(\em)$. 
	\item \textup{(Local case)} There exists $q_0 \in \em$ and an arc-length parametrized geodesic
	$\gamma$ such that $\gamma(0)=q_0$ and $\left \| \gamma''(0) \right \|  = 1/R(\em)$.
\end{itemize}
\end{theorem}

Here, $Med(\em)$ is the medial axis of $\em$, that subset of $\bbR^D$ on which the nearest point projection on $\em$ is ill-defined, namely
$$
Med(\em) = \{z\in\bbR^D ~~|~~\exists\, p,q \in \em,~p \neq q,~d(z,p) = d(z,q) = d(z,\em)\}.
$$

We say that the result above is only `almost' a dichotomy because it is possible for both conditions to hold simultaneously. The curve $\gamma$ could be one half of a circle with radius $R(\em)$ joining $q_1$ and $q_2$, for example, in which case the term `bottleneck' might be considered a misnomer, or the points $q_1$ and $q_2$ might not lie on $\gamma$ at all, so that the two assertions hold completely independently.

This situation invites us to consider two separate invariants. One, the {\it weak feature size}, $\rwfs$, is a widely studied invariant encoding large scale information such as bottlenecks. The second, which we will call the {\it local reach}, $R_{\ell}$, following \cite{Reach}, will encode curvature information.  Theorem \ref{reachdichotomy} states that the minimum of these two invariants is the reach, $$R = \min \left\lbrace R_{\ell}, \rwfs \right \rbrace.$$ 

Note that, in Riemannian geometry, the local reach is referred to as the {\it focal radius of $\em$}, while the reach itself is often referred to as the {\it normal injectivity radius of $\em$}. 

\subsection{The weak feature size}

The weak feature size is defined in terms of critical points of the distance function from $\em$ (in the sense of Grove and Shiohama; see for instance \cite{grove}, p.\ 360).

Consider the function, $d_{\em}: \rrr^D \to \rrr$ defined by $d_{\em} (y) = \inf_{p \in \em} \|y - p\|.$  Note that $\em = d_{\em}^{-1}(0)$.  Following \cite{ALS}, let $\Gamma_{\em}(y) = \{x \in \em: d_{\em}(y,\em) = \| x - y \| \}$, i.e., those $x$ in $\em$ realizing the distance between $y$ and $\em$.  Then we define a generalized gradient as
$$
\nabla_{\em}(y) := \frac{y - \text{Center}(\Gamma_{\em}(y))}{d_{\em}(y, \em)},
$$
where $\text{Center}(\sigma)$ is defined as the center of the smallest (Euclidean) ball enclosing the bounded subset $\sigma \subseteq \rrr^D$.  This generalized gradient $\nabla_{\em}$ for $d_{\em}$ coincides with the usual gradient where $d_{\em}$ is differentiable. We say that a point $y \in \rrr^D \setminus \em$ is a \textit{critical point} of $d_{\em}$ if $\nabla_{\em}(y) = \mathbf{0}$.

For example, if $y$ is the midpoint of a chord the endpoints of which meet the submanifold perpendicularly, then from $y$ there are two shortest paths to $\em$ which travel in opposite directions. It follows that $y$ is a critical point.

\begin{definition} \label{def: glob reach}
	Given a submanifold $\em$ of $\rrr^D$ let $\mathcal{C}$ denote the set of critical points of the distance function $d_{\em}$. The \textit{weak feature size}, denoted $\rwfs(\bf M)$ or simply $\rwfs$, is then defined as $\rwfs := \inf \{ d_{\em}(y) \colon y\in \mathcal{C}\}$.
\end{definition}

By Theorem \ref{reachdichotomy}, if the reach is realised globally then the first critical point will be the midpoint of a shortest chord which meets $\em$ perpendicularly at both ends, and so the weak feature size is equal to the reach.

\subsection{The local reach} In the local case, Theorem \ref{reachdichotomy} tells us that the reach is determined by the maximum value of $\| \gamma'' \|$ over all arc-length parametrised geodesics $\gamma$. 
This can be formulated more concisely by considering instead the \textit{second fundamental form}, $\Pi$, which measures how the submanifold $\em$ curves in the ambient Euclidean space $\rrr^D$.  We refer the reader to a standard text in Riemannian geometry such as \cite{docarmo}  for a precise definition of the second fundamental form.  Informally, the second fundamental form is defined as follows. For a pair of vector fields tangent to $\em$, the (Euclidean) derivative of one with respect to the other is not usually tangent to $\em$.  In fact, the tangential component is the Levi--Civita connection of the induced (Riemannian) metric on $\em$.  The normal, or perpendicular, component yields a symmetric, bilinear form, namely, the second fundamental form, denoted by $\Pi_p$.  In particular, if the norm of $\Pi_p$ is small then $\em$ is nearly flat near $p$ and if the norm is large then it is an area of high curvature.

\begin{definition} \label{def: loc reach}
	Given a submanifold $\em$ of $\rrr^D$ let $\Pi_p$ denote the second fundamental form at $p \in \em$. The the \textit{local reach} of $\em$, denoted $R_\ell(\bf M)$ or simply $R_\ell$ is the quantity 
	$$R_{\ell} = \inf_{p \in \em} \Big\{ \frac{1}{\|\Pi_p\|_{\mathrm{op}}}  \Big\}.$$
\end{definition}

We use the term `local reach' here to reflect the fact that this quantity is generated entirely by the local geometry.  In differential geometry literature the local reach is referred to as the \textit{focal radius} of the submanifold.

\section{Geometrical Framework}\label{geomframework}

We define a class of manifolds which are suitable for the task of reach estimation. This class is the same as that considered by Aamari and Levrard \cite{AL} for other problems in minimax geometric inference. The class is that of $\curlyc^k$ submanifolds, but with some additional regularity requirements. These guarantee the existence of a Taylor expansion of the embedding of the submanifold with bounded co-efficients, as well as a uniform lower bound on the reach.

\begin{definition}\label{AL-framework} (see \cite{AL})
	For two fixed natural numbers $d < D$ and for some $k \geq 3$, $\rmin > 0$, and $\mathbf{L} = (L_{\perp}, L_3, \ldots, L_k)$, we let $\ckrl$ 
	denote the set of $d$-dimensional, compact, connected submanifolds $\em$ of $\rrr^D$ such that:
	\begin{itemize}
	\item[(\textit{i})] $R(\em) \geq \rmin$;
	\item[(\textit{ii})] For all $p \in \em$, there exists a local one-to-one parametrization $\psi_p$ of the form:
	\begin{align*}
	\psi_p \co &B_{T_p \em} (0,r) \subseteq T_p \em \to \em,\\
	&v \mapsto p + v + \mathbf{N}_p(v)
	\end{align*}
		for some $r \geq \frac{1}{4 L_{\perp}}$, with
	$\mathbf{N}_p \in \curlyc^k (B_{T_p \em} (0,r), \rrr^D)$ such that
	$$\mathbf{N}_p(0) = 0, \quad \diff _0 \mathbf{N}_p = 0, \quad \left\| \diff ^2_v \mathbf{N}_p \right\|_{\mathrm{op}} \leq L_{\perp},$$
	for all $\left\| v \right\| \leq \frac{1}{4L_{\perp}}$;
	\item[(\textit{iii})] The differentials $\diff ^i_v \mathbf{N}_p$ satisfy
	$\left\| \diff ^i_v \mathbf{N}_p \right\|_{\mathrm{op}} \leq L_i$ for all $3 \leq i \leq k$ and $\left\| v \right\| \leq \frac{1}{4L_{\perp}}$.
	\end{itemize}
\end{definition}

We define subclasses of $\ckrl$ as follows, using the gap $R_{\ell} - R_{\mathrm{wfs}}$ between the weak feature size and the local reach. For fixed values of $\Rm$ and $\bL$, 
we define
$$\cM^k_0 = \big\{\em \in \ckrl|~ R_{\mathrm{wfs}}(\em) \geq R_{\ell}(\em) \big\}$$
and 
$$\cM^k_\alpha = \big\{\em \in \ckrl~|~R_{\mathrm{wfs}}(\em) \leq R_{\ell}(\em) - \alpha\big\},\;\;\alpha >0.$$
Note that 
$$\ckrl = \cup_{\alpha \geq 0} \cM^k_\alpha.$$

Manifolds in $\ckrl$ admit a second parametrization, one that represents the manifold locally as the graph of a function over the tangent space so that the first non-zero term in the Taylor expansion is of degree two and is given by the second fundamental form. These parametrizations in general satisfy weaker bounds than $\bL$. The degree $k$ Taylor polynomial then gives an algebraic approximation of the manifold, which will be very useful in later calculations.  The following lemma from \cite{AL} describes the Taylor expansion of a local parametrization at every point $p\in \em$.

\begin{lemma}\label{l:AL-param}\cite[Lemma 2]{AL} \label{lem: Wass et al}
	Let $k \geq 3$, $\em \in \ckrl$ and $r=\tfrac{1}{4}\min\lbrace \rmin, L_{\perp}^{-1}\rbrace$. Then for all $p \in \em$ there is a local one-to-one parametrization around $p$, $\Phi_p : U \to \em$, for some $U \subset T_p \em$, which contains $B(p,r) \cap \em$ in its image, satisfies $\pr_{T_p \em} \circ \Phi_p (v) = v$ on its domain, and takes the form
	\begin{align*}
	\Phi_p(v) = p + v + \frac{1}{2} T_2(v^{\otimes 2}) + \frac16 T_3(v^{\otimes 3})+\ldots+\frac{1}{(k-1)!} T_{k-1}(v^{\otimes (k-1)})+\mathcal{R}_k(v),
	\end{align*}
	where $\|\mathcal{R}_k(v)\|\leq C \|v\|^k$. Furthermore $T_2 = \Pi_p$ and $\|T_i\|_{\op} \leq L'_i$, where $L'_i$ and $C$ depends on $d$, $k$, $\rmin$ and $\mathbf{L}$, and the terms $T_2,\ldots, T_{k-1}, \mathcal{R}_k$ are all normal to $T_p \em$.
\end{lemma}

\begin{definition}\label{d:degkapprox}
We call the degree $j$ truncation of the parametrization  $\Phi_p$ given in Lemma \ref{lem: Wass et al} the {\it approximation of degree $j$ to $\em$ around $p$} and write it 
$$\Phi_p^j(v) = p + v + \frac{1}{2} T_2(v^{\otimes 2}) + \frac16 T_3(v^{\otimes 3})+\ldots+\frac{1}{j!}T_j(v^{\otimes j}).$$
\end{definition}

\section{Convexity defect functions}\label{s:cdf}

The convexity defect function, originally introduced by  Attali, Lieutier and Salinas \cite{ALS}, measures how far a subset $\ex \subseteq \rrr^D$ is from being convex at scale $t$. The goal of this section is to establish a relationship between the convexity defect function and the reach.  The definition is valid for any compact subset of $\rrr^D$.  In this section we will principally consider the case of a closed submanifold $\em$ as before, but in the sequel we will need to know that this function can be defined in greater generality.

We recall the definition.
Given a compact subset $\sigma \subseteq \ex$, it is contained in a smallest enclosing closed ball in $\rrr^D$. We define $\rad(\sigma)$ to be the radius of this ball. We denote by $\hull(\sigma)$ the convex hull of $\sigma$ in $\rrr^D$. Then we define the \textit{convex hull of $\ex$ at scale $t$} to be the following subset of $\rrr^D$:
$$
\hull(\ex, t) = \bigcup_{\substack{\sigma\subseteq \ex\\ \rad(\sigma) \leq t}} \hull(\sigma).
$$
For two compact subsets $A$ and $B$ of $\mathbf{R}^D$, we define the asymmetric distance $H(A|B) = \sup_{a \in A} d(a,B)$ so that $H(A,B) = \max \big\lbrace H(A|B), H(B|A) \big\rbrace$ is the symmetric Hausdorff distance. 
\begin{definition} \label{def: convexity defect function}
	Given a compact subset $\ex \subseteq \rrr^D$, we define the \textit{convexity defect function} $h_{\ex}: \rrr_{\geq 0} \rightarrow \rrr_{\geq 0}$ by $h_{\ex}(t) = H(\hull(\ex, t) , \ex) = H(\hull(\ex,t) ~|~\ex)$.
\end{definition}

\begin{figure}[h!]
	\includegraphics[scale = .12]{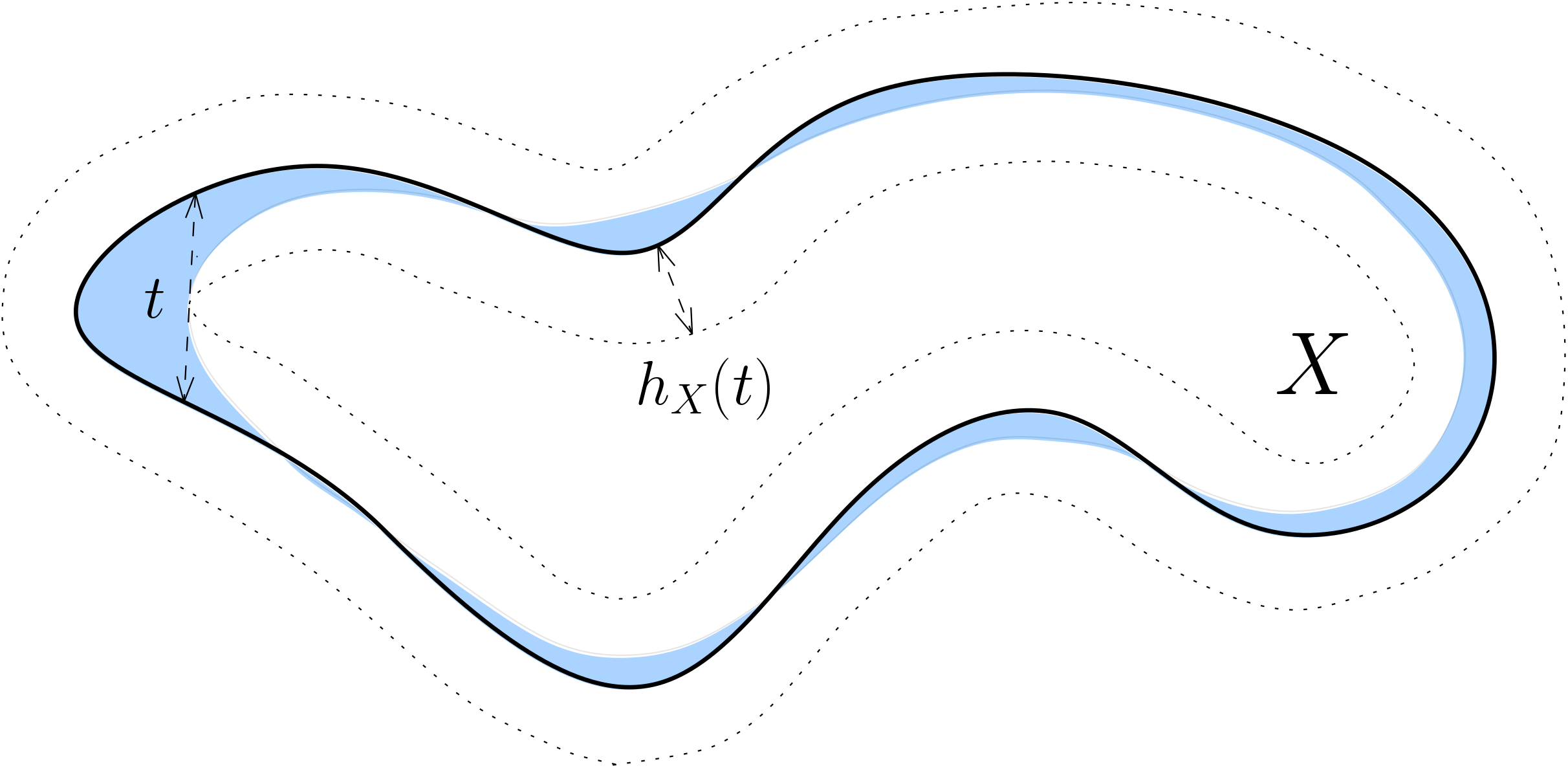}
	\caption{The convex hull at scale $t$, $\hull(\ex,t)$ (in blue), of a curve $\ex$ (in black). Enclosed between the dotted curves is  the minimal tubular neighborhood around $\ex$ that contains $\hull(\ex,t)$ --- its width is the convexity defect function $h_\ex(t)$.}
\end{figure}

 We recall here from \cite{ALS} some useful properties of $h_{\ex}$.

\begin{itemize}
	\item[1.] $h_{\ex}(0) = 0$.
	\smallskip
	
	\item[2.]	$h_{\ex}$ is non-decreasing on the interval $[0, \rad(\ex)]$ and constant thereafter.
	\smallskip
	
	\item[3.] If $\tildex \subseteq \rrr^D$ satisfies $H (\ex, \tildex) < \epsilon$, where $H$ is the Hausdorff distance, then $h_{\tildex}(t-\epsilon) - 2\epsilon \leq h_{\ex}(t) \leq h_{\tildex}(t+\epsilon) + 2\epsilon$ for any $t\geq \epsilon$.
\smallskip
	
	\item[4.] $h_{\ex}(t) \leq t$ for all $t \geq 0$. Moreover, $h_{\ex}(t_0) = t_0$ if and only if $t_0$ is a critical value of the distance function, $d_{\ex}$.
		\smallskip
	
	\item[5.] If the reach, $R = R(\ex) >0$, then on $[0,R)$ the function $h_{\ex}(t)$ is bounded above by a quarter-circle of radius $R$ centered on $(0,R)$. In other words, $h_{\ex}(t) \leq R - \sqrt{R^2 - t^2}$ for $t \in [0,R)$.
\end{itemize}

From item 4 and the definition of the weak feature size in terms of critical points of the distance function, the following proposition is immediate.

\begin{proposition}
	If $\em$ is a submanifold of $\rrr^D$ then $\rwfs = \inf \left\lbrace t >0\colon h_{\em}(t)=t\right\rbrace$.
\end{proposition}

We can also relate the local reach to the convexity defect function with the following proposition, which we will prove in Section \ref{hnearzero}.

\begin{proposition} \label{taylorpoly} Let $k \geq 4$. There exists a constant $C$ (depending on $\rmin, \mathbf{L}, d$ and $k$) such that, for any sufficiently small non-negative real $t$, $t \leq t_{\rmin, \bL,d,k}$, and any $\em \in \ckrl$, we have
	\begin{align*}
	\left|h_{\em}(t) - \frac{t^2}{2R_\ell} \right| \leq Ct^4.
	\end{align*} 
	
In case $k=3$, there exists a constant $C'$ (depending on $\rmin, \mathbf{L}, d$) such that, for any sufficiently small non-negative real $t$, $t \leq t_{\rmin, \bL,d}$, and any $\em \in \ckrl$, we have
\begin{align*}
\left|h_{\em}(t) - \frac{t^2}{2 R_\ell} \right| \leq C't^3.
\end{align*}
\end{proposition}

We will write, somewhat informally, $$R_{\ell} = 1/h''_{\em}(0).$$ The function $h_{\em}$ is not actually twice differentiable; $h''_{\em}(0)$ here is a `pointwise second derivative'.
Since $R = \min \left\lbrace R_{\ell}, \rwfs \right \rbrace$, these two propositions show how the convexity defect function yields the reach.

Proposition \ref{taylorpoly} will be proven in Section \ref{hnearzero}, but first we need to refine the upper bound given in item 5 of the properties of $h_{{\bf X}}$ given after Definition \ref{def: convexity defect function} for the case where $\ex$ is a submanifold.

\subsection{Upper bounds on the convexity defect function} \label{sec: upp bound defect}

The two aspects of the reach relate to the convexity defect function in quite different ways, which naturally leads one to wonder which aspect of the reach is responsible for item 5 of the properties of $h_{{\bf X}}$ given after Definition \ref{def: convexity defect function}.  In this subsection we improve the upper bound by increasing the radius of the bounding circle from $R$ to $R_{\ell}$, though the bound still only holds on the interval $[0,R)$ (compare with Lemma 12 in \cite{ALS}). See Figure \ref{fig:hxt} for an illustation.

\begin{proposition}\label{p:newbound}
	If $\em \in \ckrl$and $R = R(\em)$ is its reach, then on $[0,R)$ the function $h_{\em}(t)$ is bounded above by a quarter-circle of radius $R_{\ell}$ centered on $(0,R_{\ell})$.  In other words, $h_{\em}(t) \leq R_{\ell} - \sqrt{R_{\ell}^2 - t^2}$.
\end{proposition}

\begin{figure}[h!] 
\begin{minipage}{5in}
	\centering
	  \raisebox{-0.5\height}{\includegraphics[scale = .2]{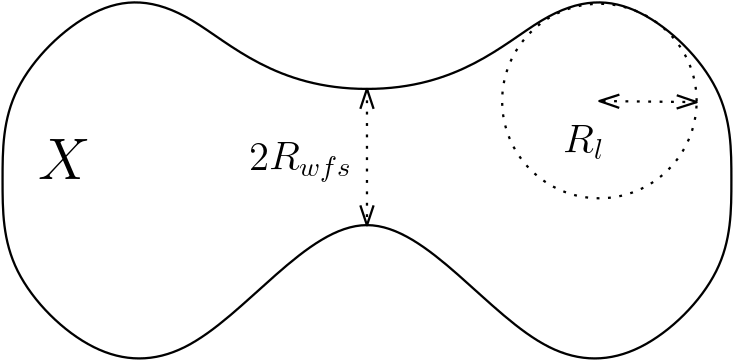}} 
		\hspace{1pt}
	  \raisebox{-0.5\height}{\includegraphics[scale = .115]{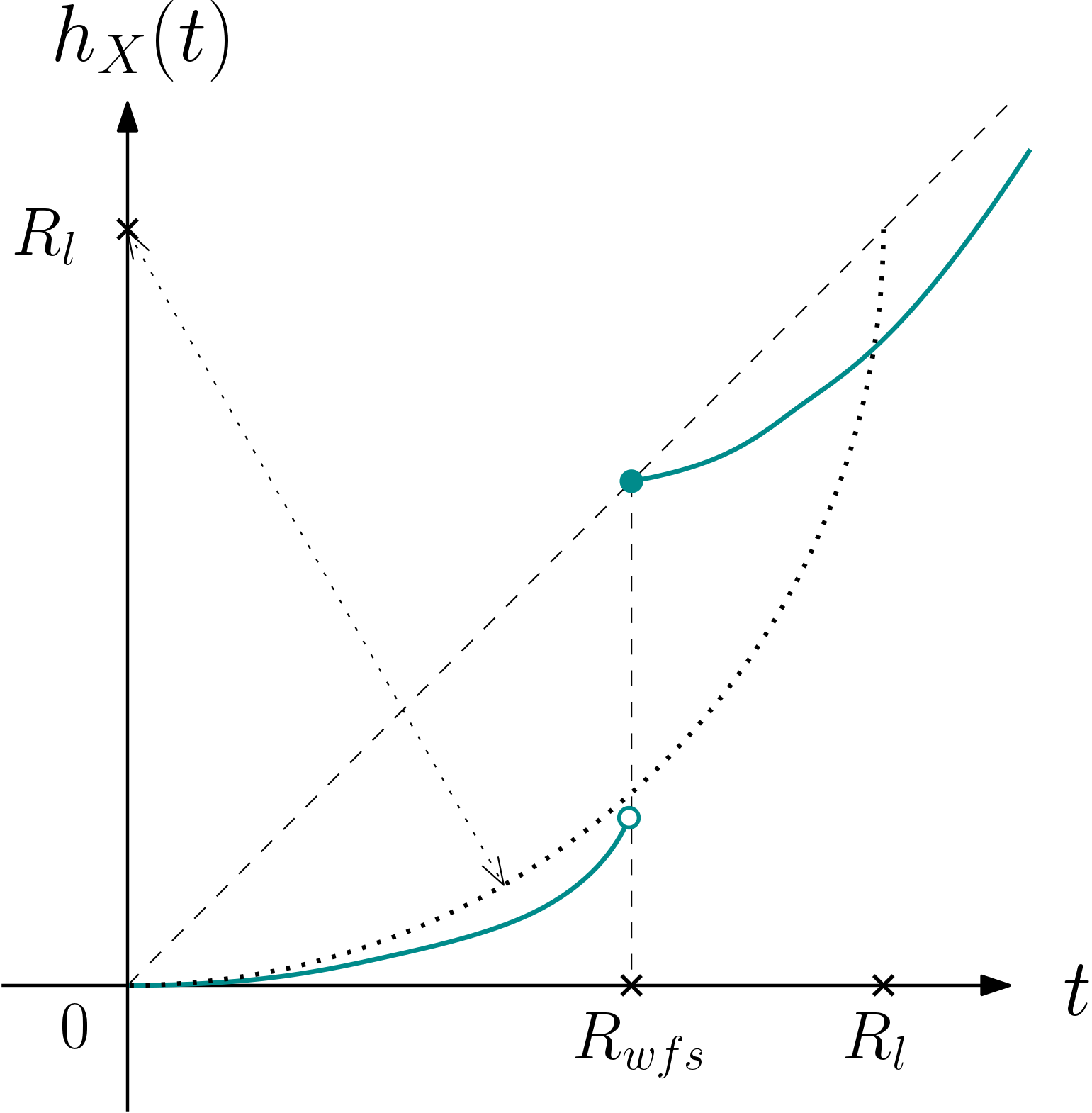}}
\end{minipage}
\caption{A curve $\ex$ (left) and its convexity defect function $h_\ex(t)$ (right), which is below the quarter-circle of radius $R_\ell$ for $t < R(\ex) = \rwfs$. Since $\rwfs < R_\ell$, we observe a discontinuity at $t = \rwfs$.}
\label{fig:hxt}
\end{figure}

For submanifolds in the class $\cM_0^k$ (where $\rwfs \geq R_{\ell}$), this result does not have any content. However, for manifolds in $\cM^k_{\alpha}$ i.e., manifolds for which $\rwfs \leq R_{\ell} - \alpha$ for some $\alpha > 0$, the bound is sharper, with the following consequence.

\begin{corollary}\label{c:discontinuity}
	If $\em \in \cM_\alpha^k$ for some $\alpha > 0$, then  $h_{\em}$ is discontinuous at $R(\em)$.
\end{corollary}

\begin{proof}
	 Since $\alpha > 0$, we have $R(\em) = \rwfs < R_{\ell}$. For $t < \rwfs$ the bound  $h_{\em}(t) \leq R_{\ell} - \sqrt{R_{\ell}^2 - t^2}$ from Proposition \ref{p:newbound} holds. On the other hand, for $t = \rwfs$ we have $h_{\em} (t) = t$. Therefore the one-sided limit $\lim_{t\nearrow \rwfs}\, h_{\em}(t) < h_{\em}(\rwfs)$ and the function is discontinuous.
\end{proof}

The proof of Proposition \ref{p:newbound} will require a few steps.  
We can focus our attention on the local reach by paying attention to sets of the form $\em' = \em \cap B(z,r)$, where $z \in \rrr^D$, $0 < r < R(\em)$ and $B(z,r)$ is a closed ball.
Lemma \ref{l:nobottlenecks} will show that subsets of this type have no bottlenecks.
We would expect, then, that the reach of such a subset is generated by the local geometry.  Lemma \ref{l:reachofemprime}  quantifies this point: the reach of $\em'$ is determined by the behaviour of the second fundamental form on $\em'$. The principal point of difficulty here relates to the boundary of the sets $\em'$.  The proposition then follows from the fact that $h_{\em}(t)$ can be bounded using the functions $h_{\em'}(t)$ and so the bound is in fact determined by the second fundamental form, i.e. by $R_{\ell}$ in particular.

\begin{lemma}\label{l:nobottlenecks}
	Let $A \subseteq \rrr^D$ be a compact set. Let $0 < s < R(A)$, $z \in \rrr^D$, and $A' = A \cap B(z, s)$, where $B$ is a closed ball. If $A' \neq \emptyset$, then $A'$ cannot have any bottlenecks, i.e. there is no pair $p,q \in A'$ with $\| p - q \| = 2R(A')$ and $(p+q)/2 \in Med(A')$.
\end{lemma}

\begin{proof}
	Suppose for a contradiction that a bottleneck exists. Then it is a chord of length $2R(A')$. Since $\diam A' \leq 2s$ we obtain that $2R(A')  \leq 2s < 2 R(A) \leq 2R(A')$, the last inequality holding by  \cite[Lemma 5]{AL15}.
\end{proof}

We now consider the case where $A= \em$, a submanifold, and consider the intersections $\em'$.
Our goal is to find the reach of the intersections, $\em'$, in order to bound $h_{\em'}$ and hence $h_{\em}$. We will use the following characterisation of the reach due to Federer \cite{federer}
$$
\frac{1}{R(A)} = \sup_{p,q \in A}  \frac{2 d(q-p, C_pA)}{\|q-p\|^2},
$$
where $C_pA$ is the tangent cone at $p$, which Federer showed always exists for a set of positive reach.
This quotient can be related to the second fundamental form as follows (cf. \cite[Lemma 3.3]{Reach}; and also work of Lytchak \cite{Lyt} for more general results).

\begin{lemma}\label{l:IInorm}
Let $k \geq 3$ and $\em \in \ckrl$.  Let $\em' = \em \cap B(z,r)$, where $z \in \rrr^D$, $0 < r < R(\em)$ and $B$ is a closed ball.  Then, provided $\em'$ contains more than a single point, for any $p \in \em'$ the norm of the second fundamental form is given by
$$
\| \Pi_p \|_{\mathrm{op}} = \limsup_{\substack{q \rightarrow p\\q \in \em'}} \frac{2 d(q-p, C_p\em')}{\|q-p\|^2},
$$
where $C_p \em'$ is the tangent cone at $p$ in $\em'$. In particular, $1/R(\em') \geq  \sup_{p \in \em'} \|\Pi_p\|_{\mathrm{op}}$.
\end{lemma}

\begin{proof}
We claim that $\partial \em'$ is a $\cC^k$ submanifold of $\em$. Consider the distance function to the central point $z \in \rrr^D$, say $f(y) = d(z,y)$.  This function is smooth on $\rrr^D \setminus z$ and its pull-back $\left. f \right| _{\em}$ is $\curlyc^k$ on $\em \setminus z$.   For any $p \in \partial \em'$, $ f(p)=r$. Note that $r$ is a critical value of $\left. f \right| _{\em}$ precisely when the distance sphere $\partial B(z,r)$ is tangent to $\em$ at some $p \in \em$.  

However, this cannot happen for $r < R(\em)$. This is because $r$ is less than the focal radius at $p$ and so $\em$ must lie in the exterior of $B(z,r)$. 
This in turn implies that $\em' = \{ p\}$ which contradicts the assumption that it is not a singleton.  Therefore, $r$ is a regular value of the $\cC^k$ function $f$ on $\em$ and the pre-image $\partial \em'$ is an embedded submanifold without boundary, as claimed. 

As a consequence, $\em'$ is an embedded submanifold of $\em$ of full dimension with boundary. The tangent cone in $\rrr^D$, $C_p \em'$, is given by $T_p \em$ for $p$ in the interior of $\em'$ and by a half-space of $T_p \em$ for $p \in \partial \em'$,
namely
$$C_p \em' = T_p \em \cap \{u\,|\,\langle p-z,u\rangle \leq 0\},$$
where $z$ is the center of the ball containing $\em'$.
We now consider some other point $q \in \em'$, $q \neq p$, and show that the projection of $q$ to $T_p\em$ lies in $C_p \em'$. Suppose $p \in \partial \em' \subseteq \partial B$. Consider the affine hyperplane $H^{D-1}$ through $p$ perpendicular to the line $pz$. Since $q \in B$, $q$ lies on the same side of $H$ as $z$ and therefore the projection of $q$ to $T_p\em$ lies in $C_p \em'$. If $p \notin \partial \em'$ then $T_p \em = C_p \em'$ and so this statement automatically holds.

Let us assume now that $q$ is close to $p$, satisfying $\|q-p\| \leq \tfrac{1}{4} \min\{R_{\text{min}}, (L_{\perp})^{-1}\}$, so that the projection of $q$ to $C_p \em'$ satisfies the conclusion of Lemma~\ref{l:AL-param}. In particular, if $v$ is the projection of $q$ onto $T_p \em$, we may write
$$
q-p = v + \tfrac{1}{2} \Pi_p(v,v) + \mathcal{R}_3(v),
$$
where the remainder $\mathcal{R}_3(v)$ is of order $O(\|v\|^3)$.  Therefore
$$
d(q-p, C_p \em') = \left\| \tfrac{1}{2} \Pi_p(v,v) + \mathcal{R}_3(v) \right\|.
$$
We can then calculate the Federer quotient,
$$
\begin{aligned}
\frac{2 d(q-p, C_p \em')}{\|q-p\|^2} &= \frac{\left\| \Pi_p(v,v) + 2 \mathcal{R}_3(v) \right\|}{ \|v\|^2 + \left\|\tfrac{1}{2} \Pi_p(v,v) + \mathcal{R}_3(v)\right\|^2}\\
	&=\frac{1}{\frac{\|v\|^2}{\left\| \Pi_p(v,v) + 2 \mathcal{R}_3(v) \right\|}  + \tfrac{1}{4} \left\| \Pi_p(v,v) + 2 \mathcal{R}_3(v) \right\|}.
\end{aligned}
$$
As $q \rightarrow p$ we see that $v \rightarrow 0$. In order to compute the $\limsup$, we may assume that a sequence of points $q_i$ is chosen such that $\|\Pi_p(v_i,v_i)\|$ is maximized. Then, since all terms in the denominator go to zero except the ratio $\frac{\|v_i\|^2}{\|\Pi_p(v_i,v_i)\|}$, we have
$$ \limsup_{\substack{q \rightarrow p\\q \in \em'}} \frac{2 d(q-p, C_p\em')}{\|q-p\|^2} = \lim_{i \to \infty} \frac{\|\Pi_p(v_i,v_i)\|}{\|v_i\|^2}.  $$
We would like to claim that $$\lim_{i \to \infty} \frac{\|\Pi_p(v_i,v_i)\|}{\|v_i\|^2}=  \| \Pi_p \|_{\mathrm{op}},$$ but recall that $p$ may lie on the boundary of $\em'$ and so we must  check that a suitable sequence of points $q_i \in \em'$ can be found.  Since $C_p \em'$ is a half-space and $\Pi_p$ is a symmetric, bilinear form, there is some unit vector $w \in C_p \em'$ so that $\|\Pi_p(w,w)\| = \|\Pi_p\|_{\mathrm{op}}$.  Then we can choose a sequence $q_i \in \em'$ so that the projections of the $q_i$ are $t_i v_i$, where the $v_i$ are unit vectors in $C_p \em'$ such that $v_i \rightarrow w$ and the $t_i$ are positive numbers with $t_i  \to 0$. The existence of such  a sequence is equivalent to the fact that $w \in C_p \em'$. The final statement then follows from
\begin{equation*}
	\| \Pi_p \|_{\mathrm{op}} = \limsup_{\substack{q \rightarrow p\\q \in \em'}} \frac{2 d(q-p, C_p\em')}{\|q-p\|^2} \leq \sup_{p,q \in \em'}  \frac{2 d(q-p, C_p\em')}{\|q-p\|^2} =
		\frac{1}{R(\em')}. \qedhere
\end{equation*}
\end{proof}

\begin{remark}
The regularity assumption of $k \geq 3$ in the previous lemma may possibly be improved to $k \geq 2$.  This stems from the assumption in Lemma~\ref{l:AL-param} which in turn derives from the regularity assumption in \cite[Lemma 2]{AL}.  However, this is not needed in the sequel so we do not pursue this further.
\end{remark}

\begin{lemma}\label{l:reachofemprime}
Let $k \geq 3$ and $\em \in \ckrl$.  Let $\em' = \em \cap B(z,r)$, where $z \in \rrr^D$, $0 < r < R(\em)$ and $B$ is a closed ball.  Then, provided $\em'$ contains more than a single point, we have $1/R(\em')=\sup_{p \in \em'} \|\Pi_p\|_{\mathrm{op}}$.
\end{lemma}

\begin{proof}
	We have already shown in Lemma \ref{l:IInorm} that $1/R(\em')\geq\sup_{p \in \em'} \|\Pi_p\|_{\mathrm{op}}$. By Lemma \ref{l:nobottlenecks}, $\em'$ does not contain any bottlenecks.  It follows that the reach is attained in one of two ways and we examine each case.
	
	\textit{Case 1}:  The reach of $\em'$ is attained by a pair of points $q,r \in \em'$ but $\| q-r \| < 2R(\em')$.  In this case we apply \cite[Lemma 3.2]{Reach} to obtain, in $\em'$, an arc of a circle of radius $R$ equal to the reach of $\em'$.  Note that that lemma is stated for manifolds, but in fact the proof only requires a set of positive reach.  Then, for any point $p$ on the reach-attaining arc, we obtain that
	$$
	\frac{1}{R(\em')} \leq \|\Pi_p\|_{\mathrm{op}} \leq \sup_{s \in \em'} \|\Pi_s\|.
	$$
	
	\textit{Case 2}:  The reach of $\em'$ is attained at a single point, say $p$, in $\em'$.  It follows, using Lemma \ref{l:IInorm} that
	$$
	\frac{1}{R(\em')}  =  \limsup_{\substack{q \rightarrow p\\q \in \em'}} \frac{2 d(q-p, C_p\em')}{\|q-p\|^2} = \|\Pi_p\|_{\mathrm{op}} \leq \sup_{s \in \em'} \|\Pi_s\|_{\mathrm{op}}.
	$$
		Combining the two cases, then, we also have that $$\frac{1}{R(\em')} \leq \sup_{s \in \em'} \|\Pi_s\|_{\mathrm{op}}$$ completing the proof.
\end{proof}

\begin{proof}[Proof of Proposition \ref{p:newbound}]
	Let $\em' = \em \cap B(z,r)$, where $z \in \rrr^D$, $0 < r < R(\em)$ and $B$ is a closed ball.
Recall that on $[0,R(\em'))$ we have 
$$h_{\em'}(t)  \leq R(\em') - \sqrt{R(\em')^2 - t^2}.$$
By Lemma \ref{l:reachofemprime}, if $\em'$ is not a single point we have 
$$\frac{1}{R_{\ell}} =  \sup_{s \in \em} \|\Pi_s\|_{\mathrm{op}} \geq \sup_{s \in \em'} \|\Pi_s\|_{\mathrm{op}} = \frac{1}{R(\em')},$$
and this entails the bound $h_{\em'}(t)  \leq R_{\ell} - \sqrt{R_{\ell}^2 - t^2}$ on $[0,R(\em'))$. If $\em'$ is a point then $h_{\em'}(t) = 0$ for all $t$ and so the same bound holds.

Recalling that $R(\em') \geq R(\em)$ for every $\em'$ with $\rad(\em')<R(\em)$, we have, for $0 < t \leq r < R(\em)$,
$$\sup_{\em' = \em \cap B(z,r)} h_{\em'}(t) \leq R_{\ell} - \sqrt{R_{\ell}^2 - t^2}.$$
Now for every $\sigma \subset \em$ with $\rad (\sigma) \leq t \leq r$, there is some $\em' = \em \cap B(z,r)$ with $\sigma \subset \em'$ and it follows that
$$h_{\em}(t) \leq \sup_{\em' = \em \cap B(z,r)} h_{\em'}(t).$$
Setting $t=r$ and combining the two inequalities, we have, for $0 < t < R(\em)$, 
\begin{equation*}
h_{\em}(t) \leq R_{\ell} - \sqrt{R_{\ell}^2 - t^2}. \qedhere
\end{equation*}
\end{proof}

\subsection{The convexity defect function near zero}\label{hnearzero} We have seen in the previous section how, for $\em \subseteq \rrr^D$ a compact submanifold, the function $h_{\em}$ on $[0,R)$ obeys an upper bound determined by $R_{\ell}$. We now study $h_{\em}$ in greater detail in a neighborhood of zero to obtain a Taylor polynomial, identifying $R_{\ell}$ as the reciprocal of the `pointwise second derivative', $1/h''_{\em}(0)$. More formally, we prove Proposition \ref{taylorpoly}, which states that, for any sufficiently small $t$,
	\begin{align*}
	\left|h_{\em}(t) - \frac{t^2}{2 R_\ell} \right| \leq Ct^{k \wedge 4}.
	\end{align*}

Once more, we approach $h_{\em}$ by considering sets $\em'$, which are the intersection of $\em$ with small closed balls. 

We introduce a new function $\hloc{\em'}$, which contains information on the convexity of $\em'$.  Lemma \ref{l:ballscenteredonem} shows how $h_{\em}$ can be determined from all the $\hloc{\em'}$. Recall from Lemma \ref{l:AL-param} that such sets $\em'$ can be written as the graphs of functions over $T_p \em$ and that these functions have Taylor expansions.

Lemma \ref{l:cdfalgebraic} will set a lower bound on $h^{\rm loc}$ for the degree 3 approximation to $\em$ around $p$, which Lemma \ref{lem5} translates to a lower bound on $\hloc{\em'}$ itself. Varying $\em'$ we obtain a lower bound on $h_{\em}(t)$ for small $t$, which we combine with the upper bound from Proposition \ref{p:newbound} to prove the result.

\begin{lemma} \label{l:ballscenteredonem}
	Let $B$ denote a closed ball. Then, for any compact set $\bf X \subset \bbR^D$ and any $r_1 ,r_2,t > 0$ satisfying $r_1 \geq 2t$ and $r_2 \geq t+r_1$, we have 
	\begin{align*}
	h_{\bf X}(t) = \sup_{p \in \ex} \hloc{\bf X}
	\end{align*}
where 
$$
 \hloc{\bf X} = \h\Big(\bigcup_{\substack{\sigma \subseteq {\bf X} \cap B(p,r_1)\\ \rad \sigma \leq t}} \hull \sigma ~ \Big  | ~{\bf X} \cap B(p,r_2) \Big).
$$
\end{lemma}
\begin{proof} We begin by showing $h_{\bf X}(t) \geq \sup_{p \in \ex} \hloc{\bf X}$.
We have immediately, for any $p \in {\bf X}$ and any $r_1,t > 0$
	\begin{align*}
	h_{\bf X}(t) = H\Big(\bigcup_{\substack{\sigma \subseteq {\bf X}\\ \rad \sigma \leq t}}
	 \hull \sigma \big| {\bf X} \Big) \geq H\Big(\bigcup_{\substack{\sigma \subseteq {\bf X} \cap B(p,r_1)\\ \rad \sigma \leq t}} \hull \sigma \big| {\bf X} \Big)
	\end{align*}
	and so all that is necessary is to check that
\begin{align*}
	 H\Big(\bigcup_{\substack{\sigma \subseteq {\bf X} \cap B(p,r_1)\\ \rad \sigma \leq t}} \hull \sigma \big| {\bf X} \Big) = 
H\Big(\bigcup_{\substack{\sigma \subseteq {\bf X} \cap B(p,r_1)\\ \rad \sigma \leq t}} \hull \sigma \big| {{\bf X}\cap B(p,r_2)} \Big) = 
\hloc{\bf X}.
	\end{align*}  Let the asymmetric distance $$\h\Big(\bigcup_{\substack{\sigma \subseteq {\bf X} \cap B(p,r_1)\\ \rad \sigma \leq t}} \hull \sigma \big| {\bf X} \Big)$$ be realized by the data $\sigma \subseteq {\bf X} \cap B(p,r_1)$, $y \in \hull \sigma$, $p' \in {\bf X}$. We have $d(p',y) \leq t$ and $d(y,p) \leq r_1$ so that $d(p',p) \leq r_1+t \leq r_2$.

Now we check that $h_{\bf X}(t) \leq \sup_{p \in \ex} \hloc{\bf X}$. If $\sigma \subset\ex$, we have $\sigma \subset B(p, 2 \rad \sigma)$ for any $p \in \sigma$. By requiring $\rad \sigma \leq t$, we obtain  $H(\hull \sigma | X) \leq  \hloc{\bf X}$ for any $p \in \sigma$ provided $r_1 \geq 2t$.
\end{proof}

For a bilinear map $S : \mathbf{R}^d \times \mathbf{R}^d \rightarrow \mathbf{R}^{D-d}$ and a trilinear map $T : \mathbf{R}^d \times \mathbf{R}^d \times \mathbf{R}^d \rightarrow \mathbf{R}^{D-d}$, we denote 
\begin{align*}
M(S,T) = \left\{(v,S(v^{\otimes 2})+T(v^{\otimes 3}))~|~v\in\mathbf{R}^d\right\} \subseteq \mathbf{R}^D
\end{align*}
which is a $\mathscr{C}^\infty$ submanifold of $\mathbf{R}^D$ of dimension $d$.

By setting $S$ and $T$ to be the coefficients of $\Phi^3_p$, the approximation of degree 3 to a manifold $\em$ around $p \in \em$ (see Definition \ref{d:degkapprox}), we can easily see that, near $p$, $M(S,T)$ is Hausdorff close to $\em$.
This assumes that $p=0$ and that $T_p \em$ is the subspace spanned by the first $d$ co-ordinates. This assumption, which is used in the statement of the lemma below, is for convenience only. For each $p \in \em$ there is an isometry of $\rrr^D$ which causes it to be satisfied.

\begin{lemma} \label{lem4}
	Let $ \em \in \ckrl$. Suppose that $p=0 \in \em$ and $T_p \em = \rrr^d \subseteq \rrr^D$. 
	
	If $k \geq 4$, we have, for $s \leq s_{1}$ with $s_1$ depending only on $\rmin, \bL, k ,d$, 
	\begin{align*}
	H\left( \em\cap B(0,s), M(S,T)\cap B(0,s)\right) \leq C s^4,
	\end{align*}
	where $S$ and $T$ are obtained from the degree $3$ approximation $\Phi_0^3$ given in Definition \ref{d:degkapprox} by $S=\frac12{\rm d}_0^2\Phi_0^3=\Pi_0$, $T=\frac16{\rm d}_0^3\Phi_0^3$ and the constant $C = C_{\rmin, \bL, k ,d}$. 
	
	When $k=3$ we can use the degree $2$ approximation $\Phi_0^2$ and pick $T \equiv 0$, to obtain
	\begin{align*}
	H\left( \em\cap B(0,s), M(S,0)\cap B(0,s)\right) \leq C' s^3
	\end{align*}
\end{lemma}
\begin{proof}
Let us initially take $s_1 = \min\lbrace \rmin, L_{\perp}^{-1}\rbrace/4$. Then for any point $q \in  \em \cap B(0,s)$, if $v = \pr_{T_0 \em}(q)$ then 
	\begin{align*}
	q=\Phi_0(v) = v + S(v^{\otimes 2}) + T(v^{\otimes 3})+\mathcal{R}(v),
	\end{align*}
	where $\Phi_0$ is the expansion given in Lemma \ref{l:AL-param}
	and $\|\mathcal{R}(v)\| \leq \frac{L'_4}{24} \|v\|^4$, unless  $k=3$.
	In  case $k=3$, if we wish to control the remainder we can only use the degree 2 polynomial approximation $\Phi_0^2$.
	
	It is therefore clear that, for the point $q = \Phi_0(v) \in \em \cap B(0,s)$, there is a corresponding point $\Phi_0^3(v) \in M(S,T)$ within the required distance and, conversely, for any point $\Phi_0^3(v) \in M(S,T) \cap B(0,s)$, there is a corresponding point $\Phi_0(v) \in \em$. The constant $C$ may be chosen to be $C = \frac{L'_4}{24}$. 
	
	However, the corresponding point is not guaranteed to lie in the ball $B(0,s)$. In the next paragraph we establish that there is a vector $v'$ very close to $v$, so that $\Phi_0^3(v')$ or $\Phi_0(v')$, as appropriate, will be sufficiently close.
	
	Let us continue to assume $k \geq 4$, since the case $k=3$ is essentially identical. We first consider the possibility that $\| \Phi_0^3(v) \| \leq s$ but $\| \Phi_0(v) \| > s$ . It is clear that, for sufficiently small $s$, $\| \Phi_0(v) \|^2 \leq s^2 + C_0 s^6$, where $C_0 $ depends on $\rmin, L_{\perp}, L_3$ and $L_4$. It follows that $\| \Phi_0(v) \| \leq s + C_1 s^5$. Consider now a vector $v' = (1-\lambda) v$, with $\lambda \approx 0$, chosen so that $\| \Phi_0(v') \| =s$. For small enough $s$ we have $\lambda \leq C_2 s^4$. It follows immediately that $\Phi_0(v')$ lies within $C_3 s^4$ of $\Phi_0(v)$, and hence within $C s^4$ of $\Phi_0^3(v)$. 
	
The case where $\| \Phi_0(v) \| \leq s$ but $\| \Phi_0^3(v) \| > s$  is dealt with similarly.
\end{proof}

The utility of $M(S,T)$ is that, since it is algebraic, we can compute explicit bounds for $h_{\ex}^{\loc}$, where $\ex = M(S,T)$.

\begin{lemma} \label{l:cdfalgebraic}Let $r_1 \leq r_2 \leq \frac{13^{1/4}}{2}\tnorm^{-1/2}$, and let $\ex =M(S,T)$.
	Then for any 
	$t \leq \min\left(
		\frac{1}{2}\snorm^{-1}, \frac{2}{\sqrt{13}}r_1
	\right)$
		we have
	\begin{align*}
	h_{\ex}^{\loc}(0,r_1,r_2 ; t) \geq \left(t - \frac12t^5 \tnorm^2\right)^2 \snorm  \geq t^2 \snorm - t^6 \snorm \tnorm^2. 
	\end{align*}
\end{lemma}

\begin{proof}
	Let  $v$ be a unit norm vector in $\mathbf{R}^d$ such that $\|S(v^{\otimes 2})\| = \snorm$. Let 
	$z \leq \min(
		\tfrac{1}{2}\snorm^{-1},\frac{2}{\sqrt{13}}r_1)$. 
	Note that the upper bound on $r_1$ gives a third upper bound for $z$, namely $z \leq 13^{-1/4} \tnorm^{-1/2}  \leq \tnorm^{-1/2}$. We set 
	$$p_1 = (zv,S((zv)^{\otimes 2}))+T((zv)^{\otimes 3})))\;\;\text{and}\;\;p_2 = (-zv,S((-zv)^{\otimes 2})+T((-zv)^{\otimes 3}))$$ 
	and denote the two-point set containing them by $\sigma = \{p_1,p_2\}$. In order to use $\sigma$ to bound $h_{\ex}^{\loc}$ we must (1) check $\sigma \subseteq \ex \cap B(0,r_1)$, (2) find the radius of $\sigma$ and (3) determine $H\left(\hull \sigma \middle| \ex \cap B(0,r_2) \right)$.
	 
	 Firstly, since $\sigma \subseteq M(S,T)$, it is enough to show that $\|p_1\|^2, \|p_2\|^2 \leq  r_1^2$. Using all three bounds on $z$, we can check
	 \begin{align*}
	 \|p_1\|^2, \|p_2\|^2 &\leq z^2+z^4 \snorm^2 + 2z^5\snorm\tnorm+z^6\tnorm^2 \\
	 &\leq 2z^2+2z^3\snorm + z^4 \snorm^2  \text{ by } z \tnorm^{1/2}<1 \\
	 &\leq \frac{13}{4} z^2 ~~ \text{ by }~ z \snorm \leq \frac{1}{2}\\
	 &\leq r_1^2  ~~\text{ by } ~z \leq \frac{2}{\sqrt{13}}r_1.
	 \end{align*}

	  Secondly, we obtain the radius as 
	 \begin{align*}
	\rad \sigma &=\frac12\sqrt{(2z)^2+(2z^3\|T(v^{\otimes 3})\|)^2}\\
	&=z\sqrt{1+z^4\|T(v^{\otimes 3})\|^2}\\
	&\leq z\left(1+\frac12z^4\tnorm^2\right) ~~\text{ since } ~~\sqrt{1+x} \leq 1+ \frac12 x~~\text{for}~x \geq 0\\
	&= z+\frac12z^5\tnorm^2.
	\end{align*}
	
	 Thirdly, we place a lower bound on $H\left(\hull \sigma \middle| \ex \cap B(0,r_2) \right)$.
	 Let $q = \frac{1}{2}(p_1+p_2) \in \hull \sigma$. For any $p = (w,S(w^{\otimes 2})+T(w^{\otimes 3})) \in \ex$ satisfying $\|w\|\leq r_2$, we have 
	\begin{align*}
	d(q,p)^2 &= \|w\|^2 + \|S(w^{\otimes 2})+T(w^{\otimes 3}) - z^2 S(v^{\otimes 2})\|^2 \\
		&\geq z^4 \snorm^2+\|w\|^2(1 - 2z^2\snorm^2 - 2z^2 r_2\snorm\tnorm).
	\end{align*}
	Since $z\snorm \leq 1/{2}$ we have $2z^2\snorm^2 \leq \frac12$. The same condition also allows us to see that $	2z^2r_2\snorm\tnorm \leq zr_2\tnorm \leq \frac12$.
	It follows that $$d(q,p)^2 \geq z^4 \|S\|^2_{\op} = d(q,0)^2$$
	from which 
	we obtain the bound
	$\h\left(\hull \sigma \middle| \ex \cap B(0,r_2) \right) \geq z^2 \|S\|_{\op}$.

	These three calculations yield 
	 $h_\ex^{\loc}(0,r_1,r_2 ; z+\frac12z^5 \tnorm^2) \geq  z^2 \snorm$. 
Now we may reparametrize the argument by setting $t=z+\frac12z^5 \tnorm^2$. Obviously $t \geq z$ so we can invert to obtain $z = t - \frac12z^5 \tnorm^2 \geq t - \frac12t^5 \tnorm^2 $
	and so $h_\ex^{\loc}(0,r_1,r_2;t) \geq  (t - \frac12t^5 \tnorm^2)^2 \snorm \geq  (t^2 - t^6 \tnorm^2) \snorm  $.
	If the bounds given in the statement hold for $t$ , then they will also hold for the smaller value $z$ and so the result is proved.
\end{proof}

We are now in a position to convert this bound for an algebraic approximation to $\em$ into one for the small patch of $\em$ itself. 

We need a stability result first.

\begin{lemma} \label{l:stabl}
Let $\bf X, Y$ be two subsets of $\bbR^D$ and let $r_1, r_2 ,t > 0$ such that $r_1 \leq r_2$. Then, if $p \in \bf X \cap \bf Y$ and $\h({\bf X} \cap B(p,r_2), {\bf Y} \cap B(p,r_2)) \leq \epsilon$ , we have 
$$
h_{\ex}^{\loc}(p,r_1,r_2 ; t) \leq h^{\loc}_{\bf Y}(p,r_1 + \epsilon, r_2 ; t + \epsilon) + 2\epsilon.
$$
\end{lemma}

\begin{proof} This is a straightforward adaptation of the proof of Lemma 5 in \cite{ALS}. Indeed let $\sigma \subset \ex \cap B(p,r_1)$ be such that $\rad \sigma \leq t$. Let $\xi = {\bf Y} \cap  B(p,r_2) \cap \sigma^\epsilon$. Since $\h(\ex \cap B(p,r_2), \ey \cap B(p,r_2)) \leq \epsilon$, $\xi$ is not empty and satisfies $\h(\xi, \sigma) \leq \epsilon$. Thus $\xi \subset \ey \cap B(p,r_1+\epsilon)$, and furthermore, by Lemma 16 in \cite{ALS}, we have $\rad\xi \leq t + \epsilon$. We conclude using that
\begin{align*}	
\hull \sigma \subset \hull (\xi^\epsilon) &= (\hull \xi)^\epsilon  \\
&\subset (\ey \cap B(p,r_2))^{h^{\loc}_{\ey} (p,r_1 +\epsilon, r_2 ; t + \epsilon) +\epsilon} \\
&\subset (\ex \cap B(p,r_2))^{h^{\loc}_{\ey} (p,r_1 +\epsilon, r_2 ; t + \epsilon) +2\epsilon}.\qedhere
\end{align*} 
\end{proof}

\begin{lemma} \label{lem5}
	Let $k \geq 4$. There exists $s_2 > 0$ depending only on $\rmin, \bL, k, d$ such that for any $r_2 \leq s_2$ and for any $r_1, t \geq 0$ such that both $r_1 \leq r_2$ and	
	$$C_0 r_2^4 \leq t \leq \frac{2}{\sqrt{13}}r_1$$ 
	for some constant $C_0$ depending on $\rmin,\bL, k ,d $,	we have, for all  $\em \in \ckrl$ and all $p \in \em$,
	\begin{align*}
	h_{\em}^{\loc}(p,r_1,r_2 ; t) \geq  \frac{1}{2}t^2 \|\Pi_p\|_{\op} - C r_2^4 
	\end{align*}
	where $C$ is a constant depending on $\rmin,\bL, k ,d $.
	
	In case $k=3$, 
	we have, for all  $\em \in \ckrl$ and all $p \in \em$,
	\begin{align*}
	h_{\em}^{\loc}(p,r_1,r_2 ; t)  \geq  \frac{1}{2}t^2 \|\Pi_p\|_{\op} - C' r_2^3 
	\end{align*}
	where $C'$ is a constant depending on $\rmin,\bL, d$.
\end{lemma}

\begin{proof}By applying an isometry of $\rrr^D$, we may assume that $p=0$ and that $T_p \em = \rrr^d \subseteq \rrr^D$. The result will then follow from Lemmata \ref{lem4} and \ref{l:cdfalgebraic}  in addition to the Hausdorff stability property for $h^{\loc}$ (Lemma \ref{l:stabl}). Take $r_2 > 0$ smaller than $s_1$ (from Lemma \ref{lem4}) and than $ \frac{13^{1/4}}{2{L'_3}^{1/2}}$ (from Lemma \ref{l:cdfalgebraic}). 
	 In the case $k \geq 4$, where $\Phi_p$ is the expansion described in Lemma \ref{l:AL-param}, $S=\frac12{\rm d}_0^2\Phi_p=\Pi_p$, $T=\frac16{\rm d}_0^3\Phi_p$ and $C_0$ is the constant from the statement of Lemma \ref{lem4}, we have
	 
	\begin{align*}
	\ h_{\em}^{\loc}(0,r_1,r_2; t) &\geq h_{M(S,T)}^{\loc}\left(0,r_1 - C_0 s^{4},r_2 ; t - C_0 r_2^4\right) - 2C_0 r_2^4 \\
	& \geq \left(t - C_0 r_2^4\right)^2 \snorm - \left(t - C_0 r_2^4\right)^6 \snorm \tnorm^2 - 2C_0 r_2^4\\
	&\geq  \frac{1}{2} \|\Pi_p\|_{\op} t^2 - C r_2^4.
	\end{align*}

where $C$ depends only on $R_{\min}, \bL, d, k$. The first inequality only holds if $C_0 r_2^4 \leq t$.
In the case $k=3$ the result is obtained similarly.
\end{proof}
We conclude with the proof of Proposition \ref{taylorpoly}.
\begin{proof}[Proof of Proposition \ref{taylorpoly}]
	By taking 
	$$t \leq \frac14 s_2 \wedge (4^4 C_0)^{-1/3} $$ 
	(from Lemma \ref{lem5}), and setting $r_1=2t$ and $r_2 = 3t$, we have that
	$$
	C_0 r_2^4 \leq t \leq \frac2{\sqrt{13}}r_1 \,\, \text{and} \,\, t+r_1 \leq r_2
	$$ 
	so that the hypotheses of both \ref{l:ballscenteredonem} and \ref{lem5} hold. It is now immediate that if $k \geq 4$
\begin{align*}
	h_{\em}(t) &= \sup_{p \in \em} h_\em^{\loc}(p,r_1,r_2;t) \\
	&\geq \sup_{p \in \em} \big(\frac{1}{2} \|\Pi_p\|_{\op} t^2 - C r_2^4\big) \\
	&= \frac{t^2}{2R_{\ell}} - 3^4 C t^4
\end{align*}
	where $C$ is a constant depending on $\rmin,\bL,d,k$,
	while if $k=3$
	\begin{align*}
	h_{\em}(t)  \geq  \frac{1}{2R_{\ell}}t^2 - C' t^3, 
	\end{align*}
	where $C'$ is a constant depending on $\rmin,\bL$.
	 On the other hand, Proposition \ref{p:newbound} provides an upper bound which will hold for all $t < \rmin$:
	\begin{align*} \label{line}
	h_\em(t) &\leq R_\ell - \sqrt{R_\ell^2 - t^2} \leq \frac{t^2}{2R_\ell} + \frac{t^4}{2R_\ell^3} \\
	&\leq \frac{t^2}{2R_\ell} + \frac{t^4}{2\rmin^3}.\qedhere
	\end{align*}
\end{proof}

\section{Approximating the reach} \label{sec: esti}

Recall item 3 of the properties of $h_{{\bf X}}$ given after Definition \ref{def: convexity defect function} which guarantees that the convexity defect function is stable with respect to perturbations of the manifold which are small in the Hausdorff distance. This allows one to approximate the reach of a submanifold $\em \subseteq \rrr^D$ from a nearby subset $\tildem$.

Given a submanifold $\em$ and another subset $\tildem$ (not necessarily a manifold) so that $H (\em, \tildem) < \epsilon$, we can calculate the convexity defect function $h_{\tildem}$.
This can then be used to approximate $R_{\ell} = \left( h_{\em}''(0) \right)^{-1}$ and $\rwfs = \inf \left\lbrace t \colon h_{\em}(t)=t, t>0 \right\rbrace$.
We can approximate the local reach via $$h_{\em}''(0) \approx 2\frac{h_{\tildem}(\Delta)}{\Delta^2}$$ for some choice of step size $\Delta$. Proposition \ref{taylorpoly} gives the following bound on the error.

\begin{proposition}\label{localloss}
	Let $\em \in \ckrl$. Let $0 < \epsilon < \Delta < 1$ be such that $\epsilon + \Delta$ is small enough to satisfy the hypotheses constraining the variable $t$ in Proposition \ref{taylorpoly}. Let $\tildem \subseteq \rrr^D$ be such that $H (\em, \tildem) < \epsilon$.
	
	Then
	\begin{itemize}
	\item If $k \geq 4$, $\left| h_{\em}''(0) - 2\frac{h_{\tildem}(\Delta)}{\Delta^2}\right| \leq A \epsilon \Delta^{-2}+ B \Delta^2$
	and, in particular, if $\Delta = \epsilon^{1/4}$,
		$$\left| h_{\em}''(0) - 2\frac{h_{\tildem}(\Delta)}{\Delta^2}\right| \leq (A+B) \epsilon^{1/2}$$
	\item If $k=3$,
	$\left| h_{\em}''(0) - 2\frac{h_{\tildem}(\Delta)}{\Delta^2}\right| \leq A \epsilon \Delta^{-2}+ B \Delta$
	and, in particular, if $\Delta = \epsilon^{1/3}$,
	$$\left| h_{\em}''(0) - 2\frac{h_{\tildem}(\Delta)}{\Delta^2}\right| \leq (A+B) \epsilon^{1/3}$$
	\end{itemize}
	where the constants $A$ and $B$ depend only on $\rmin,\mathbf{L}$.
\end{proposition}

\begin{proof}
	Set $\kappa = h_{\em}''(0)$ and $\tilde{\kappa} = 2\frac{h_{\tildem}(\Delta)}{\Delta^2}$.
	Comparing $\em$ to $\tildem$, we obtain from stability that $$2\frac{h_{\em}(\Delta-\epsilon) - 2 \epsilon}{\Delta^2} 
	\leq \tilde{\kappa} 
	\leq 2\frac{h_{\em}(\Delta+\epsilon) + 2 \epsilon}{\Delta^2}.$$
	
	In the case $k \geq 4$, Proposition \ref{taylorpoly} states that $\left| h_{\em}(t) - \frac{\kappa}{2} t^2 \right| \leq Ct^4$, for some constant $C$ depending only on $\rmin,\bL$.
	It follows that
	$$\frac{\kappa (\Delta - \epsilon )^2 - 2C(\Delta - \epsilon)^4 - 4\epsilon}{\Delta^2} 
	\leq \tilde{\kappa}
	\leq \frac{\kappa (\Delta + \epsilon )^2 + 2C(\Delta + \epsilon)^4 + 4\epsilon}{\Delta^2}. $$
	
	Expanding and using that $\epsilon, \Delta < 1$, we obtain
	\begin{align*}
	\left| \kappa - \tilde{\kappa} \right| &\leq  2C\Delta^2 + (3 \kappa + 30 C + 4 )\epsilon\Delta^{-2}.
	\end{align*}
	Similarly, in the case $k=3$, we obtain
		\begin{align*}
	\left| \kappa - \tilde{\kappa} \right| &\leq  2C'\Delta + (3 \kappa + 14 C' + 4 )\epsilon\Delta^{-2}
	\end{align*}
	where $C'$ is again a constant depending only on $\rmin,\bL$. Since $\kappa \leq 1/\rmin$, the constants may be chosen to be $A = \max \lbrace 3/\rmin + 30C + 4, 3/\rmin +14C' + 4 \rbrace$ and  $B = \max \lbrace 2C, 2C' \rbrace$. They depend only on $\rmin,\bL$.
	
	Now set $\Delta = \epsilon^p$ and seek the $p$ yielding the fastest rate of convergence of the error bound to zero. Since the exponent in the first term increases with respect to $p$ while that in the second decreases, the fastest rate is obtained by requiring the two exponents to be equal, so that $p=1/4$ for $k \geq 4$ and $p=1/3$ for $k=3$.
\end{proof}

At the weak feature size the convexity defect function satisfies $h_{\em}(t)=t$. The stability given by item 3 of the properties of $h_{{\bf X}}$ given after Definition \ref{def: convexity defect function} guarantees that the graph of $h_{\tildem}$ lies close to that of $h_{\em}$, but this alone cannot be used to approximate the first intersection of the graph of $h_{\em}$ with the diagonal. The graph of $h_{\em}$ could approach the diagonal very slowly before intersecting it, so that the error in approximating an intersection time based on the graph of $h_{\tildem}$ is not necessarily small.

However, we are only interested in approximating the weak feature size if it yields the reach, i.e. when $\rwfs < R_{\ell}$. Corollary \ref{c:discontinuity} guarantees the existence of a discontinuity in $h_{\em}$ at $\rwfs$; in this case the function $h_{\em}$ must jump at $\rwfs$ from being bounded above by a quarter circle of radius $R_{\ell}$ to intersecting the diagonal. This feature makes it possible to bound the error in an approximation. We begin with a simple lemma.

\begin{lemma}\label{l:intersection}
Fix $R > 0$. Let the intersection points of the line $y=x-6\epsilon$ and the  quarter-circle $y = R - \sqrt{R^2 - x^2}$ be $(x_0, y_0)$ and $(x_1,y_1)$. Then there is some $\epsilon_0$, which depends only on $R$, so that for $0 < \epsilon < \epsilon_0$ the bounds $x_0 \leq \frac{25}{4} \epsilon$ and $x_1 \geq  R -\frac{\epsilon}{4}$ hold. 
\end{lemma}

\begin{proof}
	The equation $x-6\epsilon = R - \sqrt{R^2 - x^2}$
	can be rearranged to give the quadratic $2 x^2 - (2 R +12 \epsilon) x + (36\epsilon + 12 R  ) \epsilon = 0$ with solutions
$$
x = \frac{2 R +12 \epsilon \pm \sqrt{(2R - 12\epsilon)^2 -288 \epsilon^2 }}{4}.
$$

For sufficiently small values of $\epsilon$, we have the bound
$$	2R -13 \epsilon 
\leq 2R -12 \epsilon - \frac{288 \epsilon^2}{4R -24\epsilon} 
\leq \sqrt{(2R - 12\epsilon)^2 -288 \epsilon^2 } $$
so that the solutions $x_0$ and $x_1$ are bounded by
\begin{align*}
x_0 &\leq \frac{2 R +12 \epsilon - (2R - 13 \epsilon)}{4} = \frac{25}{4} \epsilon\\
x_1 &\geq \frac{2 R +12 \epsilon + (2R - 13 \epsilon)}{4} = R -\frac{\epsilon}{4}.  \qedhere
\end{align*}
\end{proof}
It is clear from the proof that for any $\delta>0$ there is an $\epsilon >0$ so that the bounds can be taken to be $(6+\delta)\epsilon$ and $R_{\ell}-\delta\epsilon$. It is sufficient to proceed with $\delta = 1/4$ and so we will do so.

\begin{proposition}\label{globalloss}
	Let $\em$ be such that $R({\bf M}) > R_{\min}$ and let $\epsilon < \frac{2}{9} \rmin$ be a positive number small enough that the conclusion of Lemma \ref{l:intersection} holds for $R=\rmin$. Let $\tildem \subseteq \rrr^D$ be such that $H (\em, \tildem) < \epsilon$.
	
	Now suppose further that $\em$ is such that $R_{\ell} - \rwfs > \frac{9}{4} \epsilon$. Then the value $r =  \inf \left\lbrace t \geq \frac{22}{4} \epsilon : h_{\tildem}(t) \geq t - 3 \epsilon \right\rbrace$ satisfies the bound $\left| \rwfs - r \right| \leq \epsilon$.
\end{proposition}

\begin{proof}
	We first claim that $r \leq \rwfs + \epsilon$. To see this, suppose that $\rwfs + \epsilon < r$. Then, by the definition of $r$, either $\rwfs + \epsilon  <  \frac{22}{4} \epsilon$, which by the assumption on $\epsilon$ cannot happen, or $h_{\tildem}(\rwfs + \epsilon) < \rwfs -2\epsilon$  in which case
	$\rwfs = h_{\em}(\rwfs ) \leq h_{\tildem}(\rwfs + \epsilon) + 2 \epsilon < \rwfs$,
	which is a contradiction. 
	
	Now let us seek a lower bound for $r$, which relies on the fact that $R=\rwfs$.
	Note that $h_{\em} (r + \epsilon) \geq h_{\tildem} (r) - 2 \epsilon \geq r -5 \epsilon$.
	If the additional inequality 
	$$
	r -5 \epsilon \geq  R_{\ell} - \sqrt{R_{\ell}^2 - (r + \epsilon)^2},
	$$
	holds, so that $h_{\em} (r + \epsilon) > R_{\ell} - \sqrt{R_{\ell}^2 - (r + \epsilon)^2}$, then by Proposition \ref{p:newbound} we would have $r + \epsilon > R = \rwfs$, providing the required lower bound $r  \geq \rwfs - \epsilon$ and completing the proof.
	By Lemma \ref{l:intersection}, this additional inequality holds whenever $$\frac{25}{4} \epsilon \leq r + \epsilon \leq R_{\ell} -\frac{\epsilon}{4}.$$
	The first bound is true by the definition of $r$.
	The second follows from the upper bound for $r$ and the gap between $\rwfs$ and $R_{\ell}$: $r \leq \rwfs + \epsilon \leq R_{\ell} - \frac{5}{4} \epsilon$.
\end{proof}

\section{Minimax rates for reach estimators: Upper bounds} \label{sec: upperbounds}

Every submanifold has a natural uniform probability distribution given by its volume measure. We consider probability distributions with density bounded above and below with respect to this volume measure.  Recall the class of manifolds $\ckrl$ studied in \cite{AL}: $d$-dimensional compact, connected, submanifolds of $\rrr^D$ with a lower bound on the reach and admitting a local parametrization with bounded terms in the Taylor expansion (see Definition \ref{AL-framework}).

\begin{definition} \label{def:pk}
	For $k \geq 3$, $\rmin > 0$, $\mathbf{L} = (L_{\perp}, L_3, \ldots, L_k)$ and $0 < f_{\mathrm{min}} \leq f_{\mathrm{max}} < \infty$, we let $\mathscr{P}^k_{\rmin, \mathbf{L}}(f_{\mathrm{min}}, f_{\mathrm{max}})$ denote the set of distributions $P$ supported on some $\em \in \ckrl$ which are absolutely continuous with respect to the volume measure $\mu_{\em}$, with density $f$ 
	taking values $\mu_{\em}$-a.s. in $[f_{\mathrm{min}},f_{\mathrm{max}}]$.
\end{definition}

This will be abbreviated by $\mathscr{P}^k$ where there is no ambiguity.
We define the submodels $\cP^k_\alpha$ to be those distributions supported on elements of $\cM^k_\alpha$ (the classes defined in Section \ref{geomframework}). These submodels are such that $\cP^k = \cup_{\alpha \geq 0} \cP^k_\alpha$.

The following lemma shows that the uniform lower bound, $f_{\min}$, on the density of elements of $\mathscr{P}^k$ provides an upper bound $R_{\max}$ for both $R_\ell$ and $\rwfs$, which we will use in our estimators later in the section.

\begin{lemma}\label{lem: rmax}
There exists $R_{\max}$ depending on $d, f_{\min}, \rmin$ so that, if $P \in \mathscr{P}^k$ has support $\bf M$, then $R_{\ell}, \rwfs \leq R_{\max}$.
\end{lemma}
\begin{proof}
 Due to the relationship between curvature and volume, we have, by Point (3) on \cite[p.~2]{Alm} that $R_\ell \leq (\vol M / \omega_d)^{1/d} \leq (f_{\min} \omega_d)^{-1/d}$, where $\omega_d$ is the volume of the $d$-dimensional sphere of radius $1$.

Furthermore, Aamari and Levrard have shown \cite[Lemma 2.2]{AL18} that for some constant $C$ depending only on dimension, $\diam (\em) \leq C(d) f_{\min}^{-1} \rmin^{1-d}$. Since $\rwfs \leq \frac12 \diam(\em)$ we have $\rwfs \leq \frac12 C(d) f_{\min}^{-1} \rmin^{1-d}$. Setting $$R_{\max} := \max \lbrace (f_{\min} \omega_d)^{-1/d},  \frac12 C(d) f_{\min}^{-1} \rmin^{1-d}\rbrace,$$ we have the result.
\end{proof}

In \cite{AL} the authors construct an estimator $\hatem$ out of polynomial patches, from a sample $(X_1,\ldots, X_n)$ of random variables with common distribution $P \in \mathscr{P}^k$, supported on a submanifold $\em \in \ckrl$. That estimator has the following convergence property. (Note that the $T_i^*$ referred to below are $i$-linear maps from $T_p \em$ to $\rrr^D$ which are the $i$th order terms in the Taylor expansion of the submanifold discussed in Section \ref{geomframework}.)

\begin{theorem}[Theorem 6 in \cite{AL}]\label{ALestimator}
	Let $k \geq 3$. Set 
	$$\theta =\left(C_{d,k}\frac{\log(n)f^2_{\mathrm{max}}}{(n-1)f^3_{\mathrm{min}}}\right)^{1/d}$$ for $C_{d,k}$ large enough. If $n$ is large enough so that $0 < \theta \leq \tfrac{1}{8}\min\left\lbrace R_{\mathrm{min}}, L^{-1}_{\perp} \right\rbrace$ and $\theta^{-1} \geq C_{d,k,R_{\mathrm{min}},\mathbf{L}} \geq \sup_{2 \leq i \leq k} \left|T_i^*\right|_{\mathrm{op}} $, then with probability at
	least $1 - 2 (\frac{1}{n})^{\frac{k}{d}}$, we have
	$$H (\hatem, \em) \leq C^\star\, \theta^k$$
for some $C^\star >0$. In particular, for $n$ large enough,
	$$\sup_{P \in \mathscr{P}^k} \mathbf{E}_{P^{\otimes n}} \big[H (\hatem, \em)\big] \leq C \left( \frac{\log(n)}{n-1} \right) ^{k/d},$$
	where $C = C_{d,k,R_{\mathrm{min}},\mathbf{L},f_{\mathrm{min}},f_{\mathrm{max}}}$.
\end{theorem}

Note that the estimator is dependent on the value of $\theta \approx n^{-1/d}$ to within logarithmic terms, which serves as a bandwidth. The convergence rate of this estimator is very close to the currently established lower bound for estimating the reach $R$, which is $n^{-k/d}$; see Theorem \ref{prp:low} in Section \ref{sec: lower bounds} below.

\subsection{Estimating the local reach}

\begin{definition}
	We define an estimator for $R_{\ell}(\em)$, the local reach of a submanifold $\em$, by 
	$$\widehat R_{\ell} = \min\left\{\Big(2\frac{h_{\hatem}(\Delta)}{\Delta^2}\Big)^{-1}, R_{\max}\right\}$$ 
	where $\hatem$ is the Aamari--Levrard estimator of $\bf M$ as discussed at the beginning of Section \ref{sec: upperbounds} above, $\epsilon = C^\star \theta^k$ as in Theorem \ref{ALestimator}, $\Delta = \epsilon^{1/3}$ if $k=3$, or $\Delta = \epsilon^{1/4}$ if $k \geq 4$, and $R_{\max}$ is as in Lemma \ref{lem: rmax}. 
\end{definition}

\begin{theorem} \label{upper local condition}
	Let $k \geq 3$, let $\theta$ be as in Theorem \ref{ALestimator} and set $\epsilon = C^\star \theta^k$. 
	Then with probability at
	least $1 - 2 (\frac{1}{n})^{\frac{k}{d}}$, we have
	$$
	\big| \widehat R_{\ell} -R_\ell\big| \leq C_{d,k,R_{\mathrm{min}},\mathbf{L},f_{\mathrm{min}}} \epsilon^{1/3},
	$$	
	and, where $k \geq 4$, the exponent is $\epsilon^{1/2}$.
	Moreover,
	for $n$ large enough, we have
	$$
	\sup_{P \in \mathscr{P}^k} \mathbf{E}_{P^{\otimes n}} \big[\big| \widehat R_{\ell} -R_\ell \big|\big] \leq C \left( \frac{\log(n)}{n-1} \right) ^{\frac{k}{3d}},
	$$
	or, for $k \geq 4$, $C \left( \frac{\log(n)}{n-1} \right) ^{\frac{k}{2d}}$,
	where $C = C_{d,k,R_{\mathrm{min}},\mathbf{L},f_{\mathrm{min}},f_{\mathrm{max}}}$.
\end{theorem}

A glance at the proof shows that we actually control $\big|{\widehat R_{\ell}}^{-1}-{R_\ell }^{-1}\big|$ rather than $\big|\widehat R_{\ell} - R_\ell\big|$. This has no impact since $R_\ell  \leq R_{\max}$ is uniformly bounded and we do not seek fine control on $C$. Changing the parametrization $R  \mapsto 1/R$ in our statistical problem and estimating $1/R$ instead of $R$ would enable us to remove the projection onto $[0, R_{\max}]$ that we use to define $\widehat R_{\ell}$.

\begin{proof}
By construction, $\widehat R_{\ell} \leq R_{\max}$, and it is also clear that 
$$\Big|\frac{1}{\widehat R_{\ell}}-\frac{1}{R_\ell }\Big| \leq \Big|2\frac{h_{\hatem}(\Delta)}{\Delta^2}-\frac{1}{R_\ell }\Big|.$$
We derive
\begin{align*}
\big|\widehat R_{\ell} - R_\ell\big|  =\widehat R_{\ell}R_\ell \Big|\frac{1}{\widehat R_{\ell}}-\frac{1}{R_\ell }\Big| 
 \leq R_{\max}^2 \Big|2\frac{h_{\hatem}(\Delta)}{\Delta^2}-\frac{1}{R_\ell }\Big|.
\end{align*}
The first statement of Theorem \ref{upper local condition} is then a straightforward consequence of Proposition \ref{localloss} together with Theorem \ref{ALestimator}. Next, we have
\begin{align*}
&  \mathbf{E}_{P^{\otimes n}} \big[\big|\widehat R_{\ell} - R_\ell\big|\big] \\
 \leq &\; C_{d,k,R_{\mathrm{min}},f_{\mathrm{min}}, \mathbf{L}} \epsilon^{1/3} + 2R_{\max} P^{\otimes n}\big(\big|\widehat R_{\ell} - R_\ell\big| > C_{d,k,R_{\mathrm{min}},f_{\mathrm{min}}, \mathbf{L}} \epsilon^{1/3}\big) \\
  \leq & \; C_{d,k,R_{\mathrm{min}},f_{\mathrm{min}}, \mathbf{L}} \epsilon^{1/3} +4 R_{\max} n^{-k/d} 
\end{align*}
thanks to the first part of Theorem \ref{upper local condition}. This term is of order $(\log n/n\big)^{k/3d}$. For $k \geq 4$, we have the improvement to the exponent $\epsilon^{1/2}$ and the order becomes $(\log n/n\big)^{k/2d}$, which establishes the second part of the theorem for all values of $k \geq 3$ and completes the proof.
\end{proof}

 For $k=3,4$, then, the constructed estimator is optimal up to a $\log(n)$ factor as follows from Theorem \ref{prp:low} below. 

\subsection{Estimating the global reach}

By the earlier discussion, it is not possible to give a convergence guarantee when estimating the weak feature size, i.e. the first positive critical value of $d_{\em}$. However, in the case where $R = \rwfs$, that is, when $\rwfs < R_{\ell}$, this is possible. Accordingly, we now define an estimator for $\rwfs$ and hence an estimator for the reach itself.

\begin{definition} \label{def:rwfs}
		We define an estimator for $\rwfs$, the weak feature size of a submanifold $\em$, by
	$$\widehat R_{\mathrm{wfs}} = \min\Big\{\inf \big\lbrace t \in \rrr : \tfrac{22}{4} \epsilon < t, \enskip h_{\hatem}(t) \geq t - 3 \epsilon\rbrace,  R_{\mathrm{max}}\Big\},$$
where $\hatem$ is the Aamari--Levrard estimator of $\bf M$ as discussed at the beginning of Section \ref{sec: upperbounds} above, $\epsilon = C^\star \theta^k$ as in Theorem \ref{ALestimator} and $R_{\max}$ is as in Lemma \ref{lem: rmax}.
\end{definition}

Our estimator for the reach is then the lesser of the two individual estimators.

\begin{definition} \label{def: final est}
	Let $C^\star, \theta$ be as in Theorem \ref{ALestimator} and set $\epsilon = C^\star \theta^k$.
	We define an estimator for $R(\em)$, the reach of a submanifold $\em$, by 
	$$\widehat{R} = \min \Big\{\widehat R_{\mathrm{wfs}}, \widehat{R}_{\ell} \Big\}.$$
\end{definition}

Note that we could just as well use $\widehat R_{\ell}$ in place of $R_{\mathrm{max}}$ to cap the value of $\widehat R_{\mathrm{wfs}}$, since we do not analyse the error in the case $\widehat R_{\ell} < \widehat R_{\mathrm{wfs}}$. However, Definition \ref{def:rwfs} appears more natural as a stand-alone estimator of $\rwfs$.

\begin{theorem} \label{th: upper glob reach}
	Let $k \geq 3$, let $C^\star, \theta$ be as in Theorem \ref{ALestimator}, and set $\epsilon = C^\star\, \theta^k$, with $\epsilon$ such that
	$\tfrac{22}{4}\epsilon < \min(R_{\min}, 1)$, which is always satisfied for large enough $n \geq 1$. Then with probability at
	least $1 - 4n^{-k/d}$, we have
	$$
	\big| \widehat{R} - R\big| \leq C_{d,k,R_{\mathrm{min}},\mathbf{L}} \epsilon^{1/3},
	$$
	and, where $k \geq 4$, the exponent is $\epsilon^{1/2}$.
	In particular, for $n$ large enough,
	$$
	\sup_{P \in \mathscr{P}^k} \mathbf{E}_{P^{\otimes n}} \big[\big| \widehat{R} - R\big|\big] \leq C \left( \frac{\log(n)}{n-1} \right) ^{\frac{k}{3d}},
	$$
	or, for $k \geq 4$, $C \left( \frac{\log(n)}{n-1} \right) ^{\frac{k}{2d}}$,
	where $C = C_{d,k,\tau_{\mathrm{min}},\mathbf{L},f_{\mathrm{min}},f_{\mathrm{max}}}$.
\end{theorem}

\begin{proof} We will prove the result in three steps. In Step 1 we provide a bound in the case $\widehat R_\ell < \widehat R_{\mathrm{wfs}}$ which holds with high probability. Then in Step 2 we provide a bound in the complementary case $\widehat R_\ell \geq \widehat R_{\mathrm{wfs}}$. Finally, in Step 3, we combine the two bounds, proving the first statement, and use it to obtain the bound on the expected loss. {In the following, we use the letters $C$ and $C'$ to denote positive numbers that do not depend on $n$ and that may vary at each occurence.}

\noindent {\it Step 1)}. We have
\begin{align*}
\big|\widehat R-R\big|{\bf 1}_{\{\widehat R_\ell < \widehat R_{\mathrm{wfs}}\}} & = \big|\widehat R_\ell-\min(R_\ell, R_{\mathrm{wfs}})\big|{\bf 1}_{\{\widehat R_\ell < \widehat R_{\mathrm{wfs}}\}} \\
& \leq \big|\widehat R_\ell-R_\ell\big| +\big|\widehat R_\ell-R_{\mathrm{wfs}}\big|{\bf 1}_{(R_{\mathrm{wfs}} < R_\ell )}{\bf 1}_{\{\widehat R_\ell < \widehat R_{\mathrm{wfs}}\}} \\
& \leq 2\big|\widehat R_\ell-R_\ell\big| + \big|R_\ell-R_{\mathrm{wfs}}\big|{\bf 1}_{(R_{\mathrm{wfs}} < R_\ell )}{\bf 1}_{\{\widehat R_\ell < \widehat R_{\mathrm{wfs}}\}}
\end{align*}
by triangle inequality. For $C_1,C_2 >0$, introduce the events 
$$\Omega_1 = \big\{\big|\widehat R_\ell -R_\ell\big| \leq C_1\epsilon^{1/3}\big\}\;\;\text{and}\;\;\Omega_2 = \big\{H(\hatem,\em) \leq \epsilon\big\}.$$
On $\{\widehat R_\ell < \widehat R_{\mathrm{wfs}}\}$, we have
$$\forall t \in \big[\tfrac{22}{4}\epsilon, \widehat R_\ell\big] : h_{\hatem}(t) < t-3\epsilon,$$
therefore, on  $\{\widehat R_\ell < \widehat R_{\mathrm{wfs}}\} \cap \Omega_1$, we infer that
$$
\text{for all } t \in \big[\tfrac{22}{4}\epsilon, R_\ell-C_1\epsilon^{1/3}\big] : h_{\hatem}(t) < t-3\epsilon.
$$
By item 3 of the properties of the convexity defect function given after Definition \ref{def: convexity defect function}, on $
\Omega_2$, we have
$$
h_{\hatem}(t) \geq h_{\em}(t-\epsilon)-2\epsilon.
$$
Putting the last two estimates together, we obtain on $\{\widehat R_\ell < \widehat R_{\mathrm{wfs}}\} \cap \Omega_1 \cap \Omega_2$ the bound
$$\forall t \in \big[\tfrac{22}{4}\epsilon, R_\ell-C_1\epsilon^{1/3}\big] : h_{\em}(t-\epsilon) < t - 3\epsilon + 2\epsilon$$
or equivalently
$$\forall t \in \big[(\tfrac{22}{4}-1)\epsilon, R_\ell-C_1\epsilon^{1/3}-\epsilon\big] : h_{\em}(t) < t.$$
Therefore $h_{\em}(t)<t$ for $t \leq R_\ell-C_1\epsilon^{1/3}-\epsilon$ and this implies in turn
$$R_{\mathrm{wfs}} \geq R_\ell-C_1\epsilon^{1/3}-\epsilon.$$
We have thus proved
$$
\big|R_\ell-R_{\mathrm{wfs}}\big|{\bf 1}_{(R_{\mathrm{wfs}} < R_\ell )}{\bf 1}_{\{\widehat R_\ell < \widehat R_{\mathrm{wfs}}\}}{\bf 1}_{\Omega_1\cap\Omega_2} \leq (C_1\epsilon^{1/3}+\epsilon) \leq C\epsilon^{1/3}.
$$
Finally
$$
\big|\widehat R-R\big|{\bf 1}_{\{\widehat R_\ell < \widehat R_{\mathrm{wfs}}\}}{\bf 1}_{\Omega_1\cap\Omega_2}  \leq C\epsilon^{1/3}.
$$
\noindent {\it Step 2)}. We have
$$\big|\widehat R-R\big|{\bf 1}_{\{\widehat R_\ell \geq \widehat R_{\mathrm{wfs}}\}}  \leq T_1+T_2+T_3,$$
with
\begin{align*}
T_1 & = \big|\widehat R_{\mathrm{wfs}}-R_{\mathrm{wfs}}\big|{\bf 1}_{\big(R_{\mathrm{wfs}}+\tfrac{9}{4}\epsilon <  R_\ell \big)}{\bf 1}_{\{\widehat R_\ell \geq \widehat R_{\mathrm{wfs}}\}}, \\
T_2 & = \big|\widehat R_{\mathrm{wfs}}-R_{\mathrm{wfs}}\big|{\bf 1}_{\big(R_{\mathrm{wfs}} \leq R_\ell < R_{\mathrm{wfs}}+\tfrac{9}{4}\epsilon\big)}{\bf 1}_{\{\widehat R_\ell \geq \widehat R_{\mathrm{wfs}}\}}, \\
T_3 & = \big|\widehat R_{\mathrm{wfs}}- R_\ell\big|{\bf 1}_{(R_\ell < R_{\mathrm{wfs}})}{\bf 1}_{\{\widehat R_\ell \geq \widehat R_{\mathrm{wfs}}\}}.
 \end{align*}
By Proposition \ref{globalloss}, we have $T_1 \leq \epsilon$ on $\Omega_2$.
We turn to the term $T_2$.  We have
$$h_{\hatem}(\widehat R_{\mathrm{wfs}}) \geq \widehat R_{\mathrm{wfs}} - 3\epsilon$$
on $\{\widehat R_\ell \geq \widehat R_{\mathrm{wfs}}\}$ by construction. Thanks to item 3 of the properties of the convexity defect function given after Definition \ref{def: convexity defect function}, we also have
$$h_{\hatem}(\widehat R_{\mathrm{wfs}}) \leq h_{\em}(\widehat R_{\mathrm{wfs}} + \epsilon)+2\epsilon\;\;\text{on}\;\;\Omega_2$$
therefore
\begin{equation*} \label{eq: first bracketing T2}
\widehat R_{\mathrm{wfs}}-5\epsilon \leq h_{\em}(\widehat R_{\mathrm{wfs}} + \epsilon)
\end{equation*}
holds true on $\{\widehat R_\ell \geq \widehat R_{\mathrm{wfs}}\} \cap \Omega_2$.
 Introduce now the event
$$\Omega_3 = \big\{\widehat R_{\mathrm{wfs}}+\epsilon < R_{\mathrm{wfs}}\big\}.$$
By Proposition \ref{p:newbound}, it follows that
$$
\widehat R_{\mathrm{wfs}} -  5\epsilon \leq  R_\ell -\sqrt{R_\ell^2-(\widehat R_{\mathrm{wfs}}+\epsilon)^2}
$$
on $\{\widehat R_\ell \geq \widehat R_{\mathrm{wfs}}\} \cap \Omega_2 \cap \Omega_3$. Solving 
%\eqref{eq: quadratic} 
this inequality when $R_\ell > \widehat R_{\mathrm{wfs}}+\epsilon$ yields
$\widehat R_{\mathrm{wfs}} \geq R_\ell -C\epsilon$
for some $C>0$ that depends on $R_\ell$ only. Otherwise, $R_\ell -\epsilon \leq \widehat R_{\mathrm{wfs}}$ directly. Replacing $C$ by $\max \lbrace 1,C \rbrace$, we infer 
$$R_\ell {-C\epsilon}\leq \widehat R_{\mathrm{wfs}} \leq \widehat R_\ell \leq R_\ell+ C_1\epsilon^{1/3}$$
on $\{\widehat R_\ell \geq \widehat R_{\mathrm{wfs}}\} \cap \Omega_1 \cap \Omega_2 \cap \Omega_3$ hence
%\begin{equation} \label{}
$|\widehat R_{\mathrm{wfs}}-R_\ell | \leq {C}\epsilon^{1/3}$ on that event. Combining this estimate with the condition $|R_\ell - R_{\mathrm{wfs}}| \leq \tfrac{9}{4}\epsilon$ in the definition of $T_2$ implies 
$$\big|\widehat R_{\mathrm{wfs}}-R_{\mathrm{wfs}}\big| \leq  {C}\epsilon^{1/3}+\tfrac{9}{4}\epsilon.$$
We have thus proved
%\begin{equation} \label{eq: T2 final 1}
$$
T_2 {\bf 1}_{\bigcap_{i = 1}^3\Omega_i} \leq  {C}\epsilon^{1/3}+\tfrac{9}{4}\epsilon \leq  {C'}\epsilon^{1/3}.
$$
%\end{equation}
On the complementary event $\Omega_3^c = \{\widehat R_{\mathrm{wfs}}+\epsilon \geq R_{\mathrm{wfs}}\}$, we have, on the one hand 
%\begin{equation} \label{eq: neg omega 3}
$$
R_{\mathrm{wfs}} - \widehat R_{\mathrm{wfs}} \leq \epsilon.
$$
%\end{equation}
But on the other hand, on $\{\widehat R_\ell \geq \widehat R_{\mathrm{wfs}}\} \cap \Omega_1$, we have
\begin{align*}
\widehat R_{\mathrm{wfs}}-R_{\mathrm{wfs}} & \leq  \widehat R_\ell - R_{\mathrm{wfs}} \nonumber \\
& \leq R_\ell - R_{\mathrm{wfs}} + C_1\epsilon^{1/3} \nonumber \\
& \leq  \tfrac{9}{4}\epsilon + C_1\epsilon^{1/3} \leq  {C}\epsilon^{1/3} \label{eq: final omega 3 compl}
\end{align*}
thanks to the condition $|R_\ell - R_{\mathrm{wfs}}| \leq \tfrac{9}{4}\epsilon$ in the definition of $T_2$. Combining these bounds, we obtain 
%\begin{equation} \label{eq: T2 final 2}
$$
T_2  (1-{\bf 1}_{\Omega_3}){\bf 1}_{\Omega_1} \leq {C}\epsilon^{1/3}.
$$
%\end{equation}
Putting together this estimate and the bound $T_2{\bf 1}_{\bigcap_{i = 1}^3\Omega_i} \leq {C}\epsilon^{1/3}$ we established previously, we derive
$$T_2{\bf 1}_{\Omega_1 \cap \Omega_2} \leq {C}\epsilon^{1/3}.$$
We finally turn to the term $T_3$. On $\{\widehat R_{\mathrm{wfs}} \geq R_\ell\}$ intersected with $\{\widehat R_\ell \geq \widehat R_{\mathrm{wfs}}\} \cap \Omega_1$, we have
$$
0 < R_\ell \leq \widehat R_{\mathrm{wfs}} \leq \widehat R_\ell  \leq R_\ell+ C_1\epsilon^{1/3}
$$
which yields the estimate 
%\begin{equation} \label{eq: upper last}
$$
| \widehat R_{\mathrm{wfs}} - R_\ell | \leq C_1\epsilon^{1/3}\;\;\text{on}\;\;\{\widehat R_{\mathrm{wfs}} \geq R_\ell\} \cap \{\widehat R_\ell \geq \widehat R_{\mathrm{wfs}}\} \cap \Omega_1.
%\end{equation}
$$
Alternatively, on the complementary event $\{\widehat R_{\mathrm{wfs}} < R_\ell\}$ intersected with $\{\widehat R_\ell \geq \widehat R_{\mathrm{wfs}}\}\cap \Omega_2$ we have $\widehat R_{\mathrm{wfs}} -  5\epsilon \leq  R_\ell -\sqrt{R_\ell^2-(\widehat R_{\mathrm{wfs}}+\epsilon)^2}$ in the same way as for the term $T_2$, provided $\widehat R_{\mathrm{wfs}}+\epsilon < R_\ell$. This implies $\widehat R_{\mathrm{wfs}} \geq R_\ell {-C\epsilon}$. {Otherwise} $\widehat R_{\mathrm{wfs}}+\epsilon \geq R_\ell$ holds true. {In any event, we obtain}
%\begin{equation} \label{eq: one sided first}
$
- {C}\epsilon \leq \widehat R_{\mathrm{wfs}} -R_\ell.
$
%\end{equation}
Since 
%\begin{equation} \label{eq: one sided last}
 $\widehat R_{\mathrm{wfs}} -R_\ell \leq C_1\epsilon^{1/3}$ on $\Omega_1$, 
%\end{equation}
we { conclude}
%\begin{equation} \label{eq: combination}
$$
\big| \widehat R_{\mathrm{wfs}} -R_\ell \big| \leq \epsilon + C_1\epsilon^{1/3} \leq C\epsilon^{1/3}\;\;\text{on}\;\;\{\widehat R_{\mathrm{wfs}} < R_\ell\} \cap \{\widehat R_\ell \geq \widehat R_{\mathrm{wfs}}\} \cap \Omega_1 \cap \Omega_2.
%\end{equation}
$$
Combining these two bounds for $\big| \widehat R_{\mathrm{wfs}} -R_\ell \big|$, we finally derive
%\begin{equation} \label{eq: final T3}
$$
T_3 {\bf 1}_{\Omega_1 \cap \Omega_2} \leq C\epsilon^{1/3}.
$$
%\end{equation}
Putting together our successive estimates for $T_1, T_2$ and $T_3$, we have proved
\begin{equation*} \label{eq: final step 2}
\big|\widehat R-R\big|{\bf 1}_{\{\widehat R_\ell \geq \widehat R_{\mathrm{wfs}}\}}{\bf 1}_{\Omega_1 \cap \Omega_2}  \leq \epsilon + 2C\epsilon^{1/3}\leq C'\epsilon^{1/3}.
\end{equation*}
\noindent {\it Step 3).}  Combining Step 1) and Step 2) yields
$$\big|\widehat R-R\big|{\bf 1}_{\Omega_1 \cap \Omega_2} \leq C\epsilon^{1/3}.$$
By Theorem \ref{upper local condition}, we have $P^{\otimes n}(\Omega_1) \geq 1-2n^{-k/d}$ as soon as $C_1 \geq C_{d,k,R_{\mathrm{min}},f_{\mathrm{min}}, \mathbf{L}}$. 
%in \eqref{eq: est loc proba}. 
By Theorem \ref{ALestimator}, we have  $P^{\otimes n}(\Omega_2) \geq 1-2n^{-k/d}$. The first estimate in Theorem \ref{th: upper glob reach} follows for $k \geq 3$. The improvement in the case $k=4$ is done in exactly the same way and we omit it. 

Finally, integrating, we obtain
\begin{align*}
\mathbf{E}_{P^\otimes n}\big[\big|\widehat R-R\big|\big] & \leq C\epsilon^{1/3}+2R_{\mathrm{max}} \big(P^{\otimes n}(\Omega_1^c)+ P^{\otimes n}(\Omega_2^c)\big) \\
& \leq C\epsilon^{1/3}+4R_{\mathrm{max}}n^{-k/d} \leq C'\epsilon^{1/3} 
\end{align*}
and the second statement of Theorem \ref{th: upper glob reach} is proved for $k \geq 3$. The improvement in the case $k=4$ follows in similar fashion.
\end{proof}

\section{Minimax rates for reach estimators: Lower bounds} \label{sec: lower bounds}

We fix $\Rm$, $\bL$, $k$, $f_{\min}$ and $f_{\max}$ and recall the classes $\cP^k_\alpha$ which were defined in Section \ref{sec: upperbounds}, parametrized by the gap $\alpha
\leq R_{\ell} - \rwfs$. These sub-models are such that $\cP^k = \cup_{\alpha \geq 0} \cP^k_\alpha$. 

\begin{theorem} \label{prp:low}  If $f_{\min}$ is small enough and $f_{\max}$, $\bL$ are large enough 
(depending on $\Rm$, and on $\alpha$ for the second statement), then we have the following lower bounds on the reach estimation problem
\begin{align*} 
\liminf_{n\rightarrow \infty} n^{(k-2)/d} \inf_{\widehat R} \sup_{P \in \cP^k_0} \bbE_{P^{\otimes n}}[|\widehat R - R|] &\geq C_0 > 0  \:\:\:\:\:\:\:\:\: \text{and} \\ 
\liminf_{n\rightarrow \infty} n^{k/d} \inf_{\widehat R} \sup_{P \in \cP^k_\alpha} \bbE_{P^{\otimes n}}[|\widehat R - R|] &\geq C_{\alpha}  > 0 \:\:\:\:\:\:\:\:\: \forall \alpha > 0 
\end{align*}
with $C_0$ depending on $\Rm$ and $C_{\alpha}$ depending on $\Rm$ and $\alpha$.
\end{theorem}
In particular, the minimax rate on the whole model $\cP^k$ is of order $n^{-\frac{k-2}{d}}$. To show this proposition, we will make use of Le Cam's Lemma, restated in our context.

\begin{lemma}[Le Cam Lemma, \cite{Yu97}] \label{lem:lecam} For any two $P_1, P_2 \in \cP$, where $\cP$ is a model of manifold-supported probability measures, we have
\begin{align*}
\inf_{\widehat R} \sup_{P \in \cP} \bbE_{P^{\otimes n}}[|\widehat R - R|] \geq \frac12 |R_1 - R_2| (1 - \TV(P_1,P_2))^n,
\end{align*}
where $\TV$ denotes the total variation distance between measures and $R_1$ (respectively $R_2$) denotes the reach of the support of $P_1$ (resp $P_2$). 
\end{lemma}
Therefore, one needs to compute the total variation distance between two given manifold-supported measures. When these measures are uniform over their support, we have the following convenient formula.

\begin{lemma} \label{lem:tv} Let $M_1, M_2$ be two compact $d$-dimensional submanifolds of $\bbR^D$ and let $P_1$, $P_2$ be the uniform distributions over $M_1$ and $M_2$. Then we have
\begin{align*}
\TV(P_1,P_2) = \frac{\cH^d(M_2 \setminus M_1)}{\vol M_2} \; \; \; \; \text{if}~~~~~\vol M_2 \geq \vol M_1,
\end{align*} 
where $\cH^d$ denotes the $d$-dimensional Hausdorff measure on $\mathbb R^D$. 
\end{lemma}

\begin{proof} First note that $P_1$ and $P_2$ are absolutely continuous with respect to $\cH^d$ with densities $\frac{1}{\vol M_1} \ind_{M_1}$ and  $\frac{1}{\vol M_2} \ind_{M_2}$ respectively.  Therefore, we have the following chain of equalities.
\begin{align*}
&\TV(P_1,P_2) \\
& = \frac12 \int \left|\frac{1}{\vol M_1} \ind_{M_1} - \frac{1}{\vol M_2} \ind_{M_2} \right| \diff \cH^d \\
&= \frac{\cH^d(M_1 \setminus M_2)}{2\vol M_1} + \frac{\cH^d(M_2 \setminus M_1)}{2\vol M_2} + \frac12\cH^d(M_1 \cap M_2) \left( \frac{1}{\vol M_1} - \frac{1}{\vol M_2} \right) \\
&= \frac12 \left\{1 + \frac{\cH^d(M_2 \setminus M_1) - \cH^d(M_1 \cap M_2)}{\vol M_2}\right\} \\
&= \frac{\cH^d(M_2 \setminus M_1)}{\vol M_2}. \qedhere
\end{align*}
\end{proof}

Before proving Theorem \ref{prp:low} we need to introduce the following technical result:

\begin{lemma} \label{lem:gr}
Let $\Phi : \bbR^d \rightarrow \bbR$ be a smooth function and let $M = \{(v,\Phi(v))~~|~~v \in \bbR^d\} \subseteq \bbR^{d+1}$ be its graph. The second fundamental form of $M$ at the point $x = (v,\Phi(v)) \in M$ is given by
\begin{equation*}
\Pi_x(u,w) = \frac{\diff^2 \Phi(v)[\pr(u),\pr(w)]}{\sqrt{1 + \|\nabla \Phi(v)\|^2}},\quad \text{ for all } u,w \in T_x M
\end{equation*}
where $\pr$ is the linear projection to $\bbR^d \subseteq \bbR^{d+1}$.
\end{lemma}

\begin{proof} We define $\Psi : v \in \bbR^d \mapsto (v,\Phi(v)) \in \bbR^{d+1}$ so that $M$ is the image of $\bbR^d$ through the diffeomorphism $\Psi$. Let $x \in M$ and let $v \in \bbR^d$ be such that $x = \Psi(v)$. The tangent space $T_x M$ is given by $T_x M = \{ \diff \Psi(v)[h] = (h, \langle h, \nabla \Phi(v) \rangle )~~|~~ h \in \bbR^d\}$, so that a normal vector field on $M$ is given by
\begin{equation*}
n(x) = \left( -\frac{\nabla \Phi(v)}{\sqrt{1 + \|\nabla \Phi(v)\|^2}}, \frac{1}{\sqrt{1 + \|\nabla \Phi(v)\|^2}}\right) \in \bbR^{d+1}. 
\end{equation*}
For $u \in T_x M$, where $h = \pr u$, we have
\begin{align*}
\diff n(x)[u] = \left( - \frac{H\Phi(v)h}{\sqrt{1 + \|\nabla \Phi(v)\|^2}}, 0\right) - \frac{\langle H\Phi(v)h, \nabla \Phi(v) \rangle}{1 + \|\nabla \Phi(v)\|^2} n(x) ,
\end{align*}
where $H\Phi$ denotes the Hessian of $\Phi$.  Now for $w \in T_x M$ and $\eta = \pr w$, we have
\begin{align*}
\Pi_x(u,w) &= - \langle \diff n(x)[u], w \rangle = \left\langle  \left( \frac{H\Phi(v)h}{\sqrt{1 + \|\nabla \Phi(v)\|^2}}, 0\right),  (\eta, \langle \eta, \nabla \Phi(v) \rangle ) \right \rangle \\
&=  \left\langle   \frac{H\Phi(v)h}{\sqrt{1 + \|\nabla \Phi(v)\|^2}},  \eta \right \rangle = \frac{\diff^2 \Phi(v)[h,\eta]}{\sqrt{1 + \|\nabla \Phi(v)\|^2}} 
\end{align*}
concluding the proof. 
\end{proof}

We are now ready to prove Theorem \ref{prp:low}.

\begin{proof}[Proof of Theorem \ref{prp:low}] 
\noindent {\it Step 1: The case of $\cP^k_0$.} Let $M$ be the $d$-dimensional sphere in $\bbR^{d+1}$ of radius $r$ centered at $- r e_{d+1}$, where $e_{d+1} = (0, \ldots, 0, 1)$. We choose $r$ to be such that $r \geq 2\Rm$. Since $M$ is smooth, there exists $\bL^* \in \bbR^{k-2}$ (depending on $r$) such that $M \in \mathfrak{C}^k_{r,\bL^*}$ and thus the uniform probability $P$ on $M$ is in $\cP^k_{r,\bL^*}(a^*,a^*)$ (see Definition \ref{def:pk}) 
with $a^* = (r^d s_d)^{-1}$ and $s_d$ being the volume of the unit $d$-dimensional sphere.  

Let us now perturb $M$ to $M_\gamma$, as illustrated in Figure \ref{fig:0}.
Define for any $\gamma > 0$
\begin{equation}
\Phi_\gamma : \begin{cases} \bbR^{d+1} \rightarrow \bbR^{d+1} \\
z \mapsto z + \gamma^k \Psi(z/\gamma)e_{d+1}, \nonumber
\end{cases}
\end{equation}
where $\Psi(z) = \psi(\|z\|)$ and where $\psi : \bbR \rightarrow \bbR$ is a smooth, even, non-trivial, positive map supported on $[-1,1]$, decreasing on $[0,1]$, and with $\phi''(0) < 0$. The above map is a global diffeomorphism as soon as $\gamma^{k-1} \|\diff \Psi \|_{\op,\infty} < 1$. Moreover, we have $\|\diff \Phi_\gamma-I_D\|_{\op,\infty} = \gamma^{k-1}   \|\diff \Psi \|_{\op,\infty}$ and $\|\diff^j \Phi_\gamma\|_{\op,\infty} \leq \gamma^{k-j} \|\diff^j\Psi\|$, so that, provided $\|\diff^k \Psi\|$ is chosen small enough (depending on $r$) and that $\gamma$ is small enough (depending again on $r$), then we can apply Proposition A.5 from the supplementary material in \cite{AL} to show that the submanifold $M_\gamma = \Phi_\gamma(M)$ is in $\mathfrak{C}^k_{r/2, 2\bL^*}$. 
\begin{figure}[ht!]
\centering
\includegraphics[scale=.11]{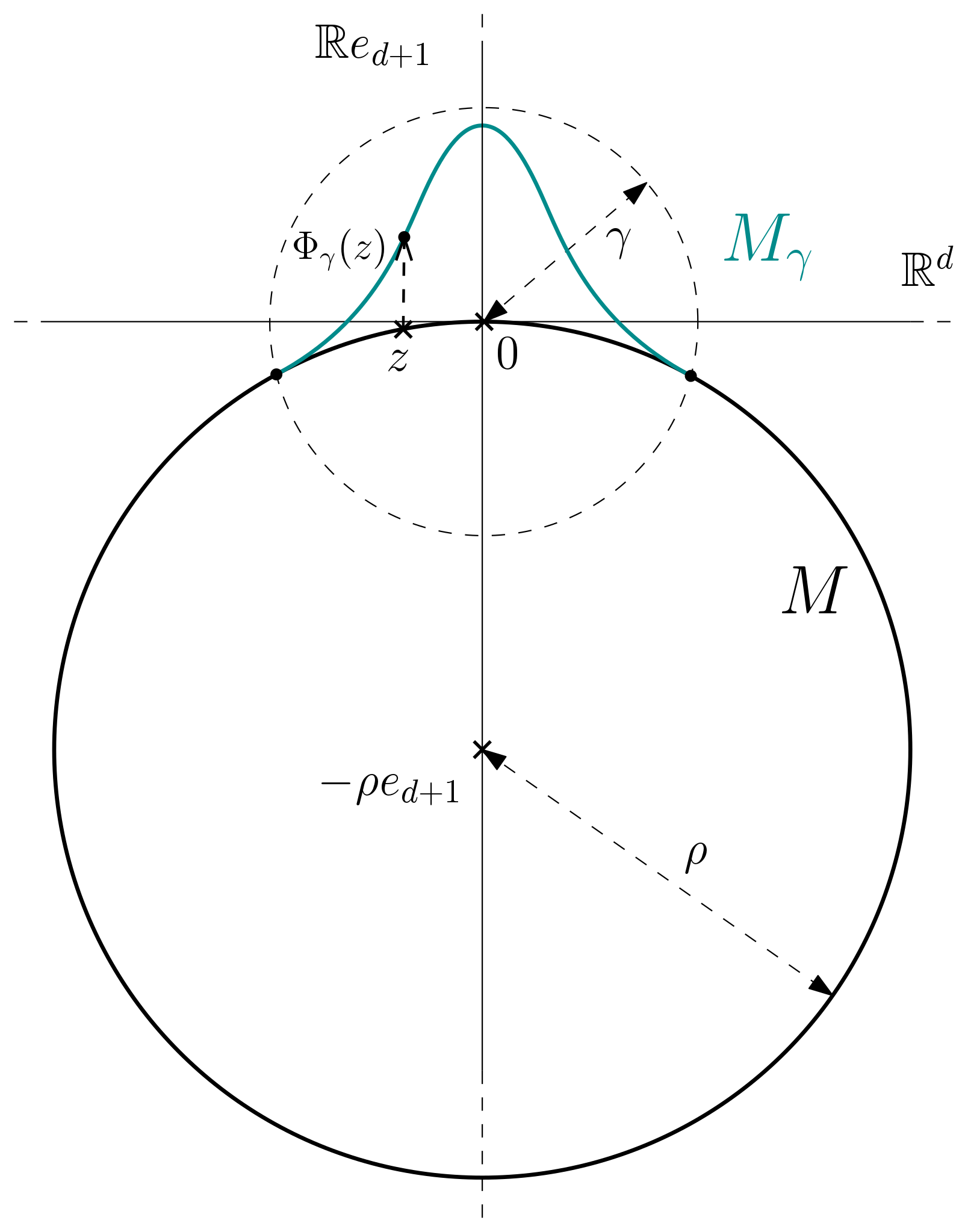}
\captionsetup{justification =  justified, margin = 1.1cm}
\caption{\small The submanifolds $M$ and $M_\gamma$ used in the proof of the first part of the lower bound.}
\label{fig:0}
\end{figure}

Then we have
\begin{equation}
\vol M_\gamma = \int_{M_\gamma} \diff \vol_{M_\gamma}(x) = \int_M |\det \diff \Phi_\gamma(z)|\, \diff\vol_M(z) \nonumber.
\end{equation}
Since $\frac12 \leq |\det d\Phi_\gamma| \leq 2$ for $\gamma$ small enough (depending on $r$), it follows that $\frac12 \vol M \leq \vol M_\gamma \leq 2 \vol M$ for the same values of $\gamma$, so that the uniform distribution $P_\gamma$ on $M_\gamma$ in is $\mathcal{P}^k_{r/2,2\mathbf{L}^*}(a^*/2,2a^*)$.
If we assume that $2\bL^* \leq \bL$, $f_{\min} \leq a^*/2$ and $2a^* \leq f_{\max}$ (which we do from now on) then we immediately have $P \in \cP^k_0$ and $P_\gamma \in \cP^k_0$, provided that $R_{\mathrm{wfs}}(M_\gamma) \geq R_{\ell}(M_\gamma)$.  We claim that the latter inequality holds.  

Around $0$, simple geometrical considerations show that $M_\gamma$ can be viewed as the graph of the function
\begin{equation}
\xi_\gamma : \begin{cases} \bbR^d \rightarrow \bbR \\
v \mapsto \sqrt{r^2 - \|v\|^2} - r + \gamma^k \psi\left(\frac{r}{\gamma}\sqrt{2 -2\sqrt{1 -\|v\|^2/r^2}}\right). \nonumber
\end{cases}
\end{equation}
Writing $\xi_\gamma(v) = \zeta_\gamma(\|v\|)$ with $\zeta_\gamma : \bbR \rightarrow \bbR$, a series of computations shows that
\begin{align*}
\zeta''_\gamma(0) = - \frac1r + r \gamma^{k-2} \psi''(0).
\end{align*}
Setting $c = -\psi''(0) > 0$ (which depends on $r$) we have, using 
Lemma \ref{lem:gr}, 
\begin{align*}
R_{\ell}(M_\gamma) \leq \frac{1}{|\zeta''_\gamma(0)|} = \frac{1}{\frac1r + c r \gamma^{k-2}} \leq r - \frac12 c r^2 \gamma^{k-2}
\end{align*}
as soon as $c r^2 \gamma^{k-2} \leq 1$. Now let us turn to the control of $R_{\mathrm{wfs}}(M_\gamma)$. We will show that the distance between any pair of bottleneck points is bounded below by $2r$. Let $(x,y) \in M_\gamma$ be a pair of bottleneck points. First notice that $x$ and $y$ cannot lie simultaneously in $B(0,\gamma)$ because $M_\gamma\cap B(0,\gamma)$ can be seen as a graph. If $x, y \in M_\gamma \setminus B(0,\gamma)$, then $d(x,y) = 2 r$ necessarily.  If, say, $x \in B(0,\gamma)$ and $y \in M_\gamma \setminus B(0,\gamma)$, then the open segment $(x,y)$ cross $M$ at a single point $z \in M$.  Therefore, we have that  $d(x,y) = d(x,z) + d(z,y)$. But now since $[x,y]$ is normal to $M_\gamma$ at point $y$, we know that $[z,y]$ is a diameter of $M$ so that $d(z,y) = 2r$ and thus $d(x,y) \geq 2r$.  We have shown that $R_{\mathrm{wfs}}(M_\gamma) \geq r \geq R_{\ell}(M_\gamma)$ for $\gamma$ small enough and thus $M_\gamma \in \cM^k_0$ and $P_\gamma \in \cP^k_0$. 

Now, by Lemma \ref{lem:tv}, we have that $\tv(P,P_\gamma) = \cH^d(M_\gamma\setminus M)/\vol M_\gamma \leq C \gamma^d$ for some constant $C$ depending on $r$. Applying now Le Cam's Lemma (Lemma \ref{lem:lecam}) and noting that $R(M) - R(M_\gamma) \geq cr^2 \gamma^{k-2}$, we obtain
\begin{align*}
\inf_{\widehat R} \sup_{P \in \cP^k_0} \bbE_{P^{\otimes n}}[|\widehat R - R|] \geq \frac12 cr^2 \gamma^{k-2} \times (1 - C \gamma^d)^n.
\end{align*}
Setting $\gamma = 1/(Cn)^{1/d}$, we know that for $n$ large enough (depending on $r$), we have
\begin{align*}
\inf_{\widehat R} \sup_{P \in \cP^k_0} \bbE_{P^{\otimes n}}[|\widehat R - R|] \geq \frac18 cr^2 (Cn)^{-(k-2)/d}.
\end{align*}
Set $r$ to be equal to $2\Rm$ and the first statement of Theorem \ref{prp:low} follows.\\  

\noindent {\it Step 2: The case of $\cP^k_\alpha$.} We next turn to the second part of the theorem.
We fix $\alpha > 0$ and construct a manifold $M \in \mathfrak{C}^k$ as follows. We consider the two parallel disks $B(0,2r) \subseteq \bbR^d \subseteq \bbR^{d+1}$ and $B(2r e_{d+1},2r) \subseteq 2r e_{d+1} + \bbR^d \subseteq \bbR^{d+1}$, with $r \geq 2\Rm$, and link them together so that $M$ satisfies the following:
\begin{itemize}
\item $M$ is a smooth submanifold of $\bbR^{d+1}$,
\item $M$ has reach $r$, and $(0,2r e_{d+1})$ is a reach attaining pair,
\item $R_{\ell}(M) \geq r + \alpha$.
\end{itemize} 
See Figure \ref{fig:a} for a schematic notion of such $M$, visualized with $d=1$.

\begin{figure}[ht!]
\centering
\includegraphics[scale=.11]{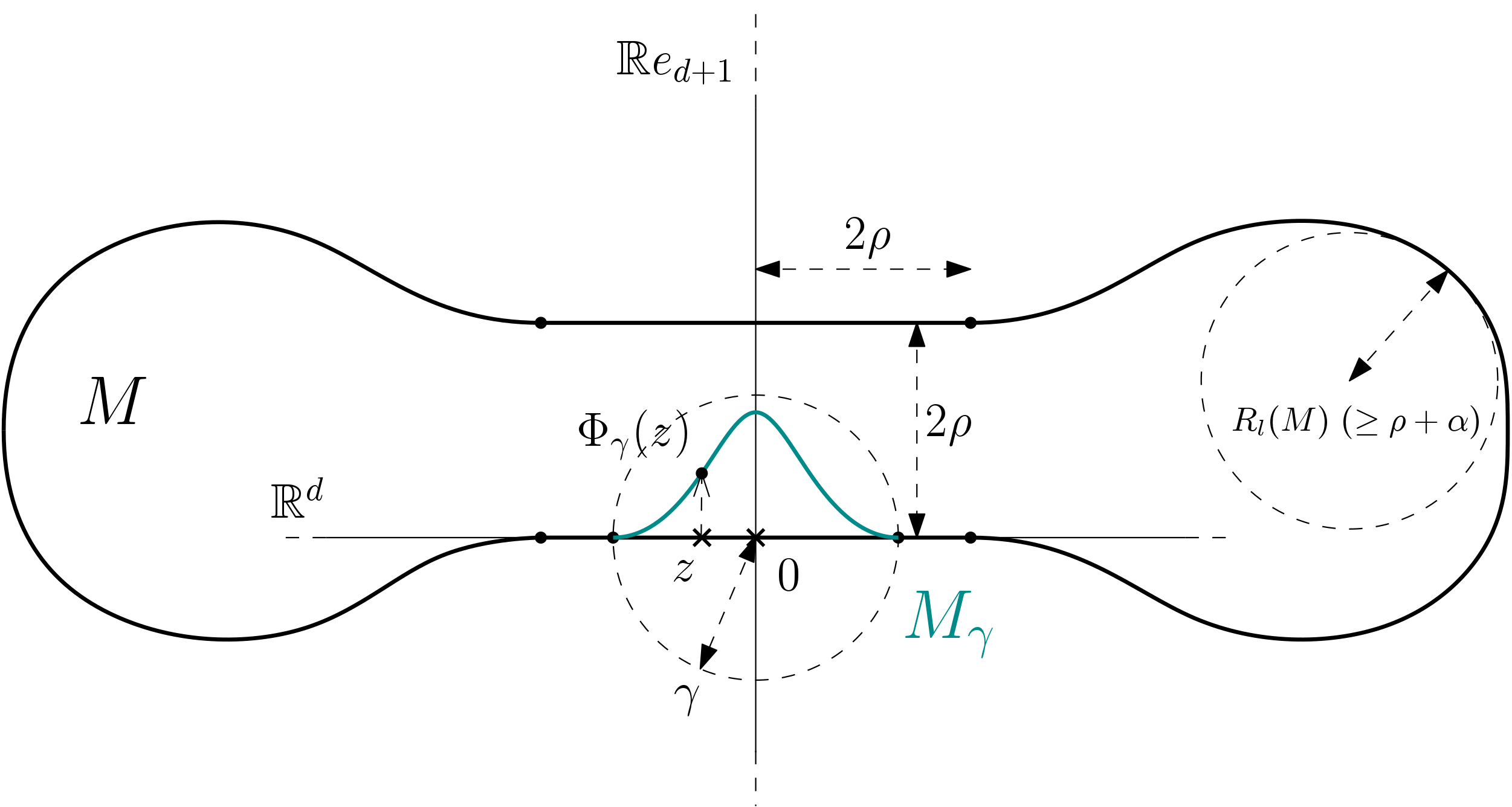}
\captionsetup{justification =  justified, margin = 1.1cm}
\caption{\small The submanifolds $M$ and $M_\gamma$ used in the proof of the second part of the lower bound.}
\label{fig:a}
\end{figure}

Furthermore, we know that there exists  $\bL^*$ (depending on $r$ and $\alpha$) such that $M \in \mathfrak{C}^k_{r, \bL^*}$ and $P \in \cP^k_{r,\bL^*}(a^*, a^*)$ where $a^* = 1/\vol M$ and where $P$ is the uniform probability over $M$. We again consider the map
\begin{equation*}
\Phi_\gamma : \begin{cases} \bbR^{d+1} \rightarrow \bbR^{d+1} \\
z \mapsto z + \gamma^k \Psi(z/\gamma)e_{d+1}.
\end{cases}
\end{equation*}
Similarly to the first part of the theorem, for $\gamma$  small enough (depending on $\alpha$ and $r$), we know that $M_\gamma = \Phi_\gamma(M)$ is a smooth submanifold in $\mathfrak{C}^k_{r/2, 2\bL^*}$ and that the uniform distribution $P_\gamma$ over $M_\gamma$ lies in $\cP^k_{r/2,2\bL^*}(a^*/2,2a^*)$. Again, assuming that $\bL \geq 2\bL^*$, $f_{\min} \leq a^*/2$ and $f_{\max} \geq 2a^*$, we have that $P \in \cP^k_\alpha$ and, furthermore, that $P_\gamma \in \cP^k_\alpha$, provided that $R_{\ell}(M_\gamma) \geq R_{\mathrm{wfs}}(M_\gamma) + \alpha$.  We claim that the latter inequality holds.

Since $\Psi$ is maximal at $0$, we know that $(\gamma^k \psi(0) e_{d+1}, 2r e_{d+1})$ is still a bottleneck pair, and thus $R_{\mathrm{wfs}}(M_\gamma) \leq r - c \gamma^k$ where we set $c = -2\psi(0)$ (depending on $\alpha$ and $r$). For the curvature, notice that it is unchanged outside of $B(0,\gamma)$ and that $M_\gamma$ is just the graph of $v \mapsto \gamma^k \Psi(v/\gamma)$ within this ball. Using Lemma \ref{lem:gr}, we thus have $R_{\ell}(M_{\gamma}) \geq \min \left\lbrace (r + \alpha) , (C\gamma^{k-2})^{-1} \right\rbrace$, with $C$ depending on $\alpha$ and $r$, so that $R_{\ell}(M_\gamma) \geq R_{\mathrm{wfs}}(M_\gamma) + \alpha$ for $\gamma$ small enough (depending on $\alpha$ and $r$), and therefore $M_\gamma \in \cM^k_\alpha$ and $P_\gamma \in \cP^k_\alpha$. 

Using Lemma \ref{lem:tv}, we have that $\tv(P,P_\gamma) = \cH^d(M_\gamma\setminus M)/\vol M_\gamma \leq \delta \gamma^d$ for some constant $\delta$ depending on $r$. Applying now Le Cam's Lemma (Lemma \ref{lem:lecam}) and noticing that $R(M) - R(M_\gamma) \geq c \gamma^{k}$, we get
\begin{align*}
\inf_{\widehat R} \sup_{P \in \cP^k_0} \bbE_{P^{\otimes n}}[|\widehat R - R|] \geq \frac12 c \gamma^{k} \times (1 - \delta \gamma^d)^n.
\end{align*}
Setting $\gamma = 1/(\delta n)^{1/d}$, we know that for $n$ large enough (depending on $r$ and $\alpha$), we have
\begin{align*}
\inf_{\widehat R} \sup_{P \in \cP^k_0} \bbE_{P^{\otimes n}}[|\widehat R - R|] \geq \frac18 c (\delta n)^{-k/d}.
\end{align*}
Setting $r = 2\Rm$ yields the result completing the proof of Theorem \ref{prp:low}.
\end{proof}

\noindent {\bf Acknowledgments}  It is a pleasure to thank the University of Oklahoma and the University of Paris--Dauphine for providing ideal working conditions, and for their support. J.\ Harvey was supported by a Daphne Jackson Fellowship sponsored by the U.K.\ Engineering and Physical Sciences Research Council and Swansea University. K.\ Shankar was  supported by the U.S.\ National Science Foundation during the completion of this work.  Any opinion, findings, and conclusions or recommendations expressed in this material are those of the authors and do not necessarily reflect the views of the National Science Foundation. The authors express their gratitude to the referees for the great care and attention shown to the manuscript, which has greatly improved the exposition.

\end{document}